\documentclass[12pt]{amsart}
\usepackage{amsmath}
\usepackage{amsfonts}
\usepackage{amssymb,color,enumitem}
\usepackage{hyperref}
\usepackage{amsthm,doi}
\usepackage{thmtools}
\usepackage{graphicx}
\usepackage{subcaption}
\usepackage{tikz}
\usepackage{pgfplots}
\pgfplotsset{compat=1.16}
\usepackage{algorithm}
\usepackage{algorithmic}
\usepackage{soul}
\usepackage{booktabs}
\usetikzlibrary{shapes.misc}
\tikzset{cross/.style={cross out, draw=black, minimum size=2*(#1-\pgflinewidth), inner sep=0pt, outer sep=0pt},
	cross/.default={1pt}}
\definecolor{octavePink}{RGB}{255,0,255}

\hypersetup{
	colorlinks   = true, 
	urlcolor     = blue, 
	linkcolor    = blue, 
	citecolor   = red 
}

\usetikzlibrary{patterns}

\definecolor{antiquefuchsia}{rgb}{0.57, 0.36, 0.51}

\newcommand{\abs}[1]{\ensuremath{\left| #1 \right| }}
\newcommand{\bigabs}[1]{\ensuremath{\big\lvert #1 \big\rvert}}
\newcommand{\card}[1]{\lvert#1\rvert}

\newcommand{\cons}{\kappa}

\newcommand{\beconst}{\gamma}
\newcommand{\conf}{{\mathrm{A}}}
\newcommand{\noise}{\sigma}
\newcommand{\tol}{\theta}
\newcommand{\was}{\mathrm{W}}

\newcommand{\grid}{\Lambda}
\newcommand{\gridl}{\Lambda_L}
\newcommand{\gridlmo}{\Lambda_{L-1}}
\newcommand{\gridlmoplus}{\Lambda_{L+2\delta-1}}

\newcommand{\gridlplus}{\Lambda_{L+2\delta}}
\newcommand{\domain}{\Omega_L}

\newcommand{\domaintol}{\Omega_{L-\tol}}
\newcommand{\domainplus}{\Omega_{L+1}}
\newcommand{\domainplusdelta}{\Omega_{L+2\delta}}
\newcommand{\domainminus}{\Omega_{L-2\delta}}
\newcommand{\domainminusone}{\Omega_{L-1}}
\newcommand{\gp}{\lambda}
\newcommand{\gpp}{\mu}
\newcommand{\outset}{\mathrm{Z}}
\newcommand{\outseta}{\mathrm{Z}_1}
\newcommand{\logi}{\log(1/\delta)}

\newcommand{\map}{\Phi}
\newcommand{\dm}{dm}

\newcommand{\alstep}[1]{\noindent {\bf #1}}

\def\E{\mathbb{E}}
\newcommand{\Var}{\operatorname{Var}}
\newcommand{\Cov}{\operatorname{Cov}}
\def\bC{\mathbb{C}}
\def\bt{F}
\def\genw{\varphi}
\def\discSTFT{\widehat{H}}
\def\discBargmann{\widehat{F}}
\def\symbOrder{\preccurlyeq}

\newcommand{\zdm}[1]{\widehat{Z}_{#1}^{\delta}}
\newcommand{\zr}[2]{\widetilde{Z}_{#1}^{#2}}
\newcommand{\zd}[2]{\widehat{Z}_{#1}^{#2}}
\newcommand{\ei}[3]{\widehat{\rho}(#1, #2, #3)}
\newcommand{\einz}[3]{\widehat{\beta}(#1, #2, #3)}

\newcommand{\subs}{\mathcal{S}}

\newcommand{\sm}[0]{S1}

\newtheorem{lemma}{Lemma}[section]
\newtheorem{theorem}[lemma]{Theorem}

\newtheorem{prop}[lemma]{Proposition}

\theoremstyle{definition}

\newtheorem{rem}[lemma]{Remark}
\numberwithin{equation}{section}

\numberwithin{equation}{section}

\author[L. A. Escudero]{Luis Alberto Escudero}
\address{L. A. Escudero, Acoustics Research Institute, Austrian Academy of Sciences,
  Vienna, Austria}
\email{lescudero@kfs.oeaw.ac.at}

\author[N. Feldheim]{Naomi Feldheim}
\address{N. Feldheim, Bar-Ilan University, Ramat-Gan, Israel}
\email{naomi.feldheim@biu.ac.il}

\author[G. Koliander]{G\"{u}nther Koliander}
\address{G. Koliander, Faculty of Mathematics, University of Vienna,
  Vienna, Austria\\
	and\\
	Acoustics Research Institute, Austrian Academy of Sciences, 
  Vienna, Austria}
\email{gkoliander@kfs.oeaw.ac.at}
	
\author[J. L. Romero]{Jos\'{e} Luis Romero}
\address{J. L. Romero, Faculty of Mathematics, University of Vienna,
  Vienna, Austria\\
	and\\
	Acoustics Research Institute, Austrian Academy of Sciences, 
  Vienna, Austria}
\email{jose.luis.romero@univie.ac.at}

\title[Efficient computation of the zeros of the Bargmann transform]{Efficient computation of the zeros of the Bargmann transform under additive white noise}

\thanks{L.\ A.\ E., G.\ K., and J.\ L.\ R.\ gratefully acknowledge support from the Austrian Science Fund (FWF): Y 1199 and P 29462. N.\ F.\ gratefully acknowledges support from the Israel Science Foundation (ISF) grant no.\ 1327/19.}



\begin{document}

\begin{abstract}
We study the computation of the zero set of the Bargmann transform of a signal contaminated with complex white noise, or, equivalently, the computation of the zeros of its short-time Fourier transform with Gaussian window. We introduce the \emph{adaptive minimal grid neighbors} algorithm (AMN), a variant of a method that has recently appeared in the signal processing literature, and prove that with high probability it computes the desired zero set. More precisely, given samples of the Bargmann transform of a signal on a finite grid with spacing $\delta$, AMN is shown to compute the desired zero set up to a factor of $\delta$ in the Wasserstein error metric, with failure probability $O(\delta^4 \log^2(1/\delta))$. We also provide numerical tests and comparison with other algorithms.
\end{abstract}

\maketitle

\section{Introduction}
\subsection{The Bargmann transform and its zeros}

The Bargmann transform of a real variable function $f \in L^2(\mathbb{R})$ is the entire function
\begin{align}\label{eq_bar}
F (z) = \big(\tfrac{2}{\pi}\big)^{\frac{1}{4}} \, e^{{-z^2}/{2}} \,
\int_{\mathbb{R}} f(t) e^{-{t^2}+2tz} \, dt, \qquad z \in \mathbb{C}.
\end{align}
Originally introduced as a link between configuration and phase space in quantum mechanics \cite{MR157250}, the Bargmann transform 
was later recognized as a powerful tool in signal analysis \cite{MR924682} because it encodes the correlations between the signal~$f$ and the time-frequency shifts of the Gaussian function $g(t) = \big(\tfrac{2}{\pi}\big)^{\frac14} \,e^{-t^2}\,$:
\begin{align}\label{eq_stft}
e^{-i x y} e^{-\frac12(x^2+y^2)}
F(x-iy) = \int_{\mathbb{R}} f(t) \overline{g(t-x) e^{2i t y}}
\, dt, \qquad x,y \in \mathbb{R}.
\end{align}
In the jargon of time-frequency analysis, the right-hand side of \eqref{eq_stft} is known as the \emph{short-time Fourier transform} of $f$ with Gaussian window, and measures the contribution to $f(t)$ of the frequency $y$ near $t=x$. 

In practice, the values of the short-time Fourier transform of a signal~$f$ are only available on (a finite subset of) a grid
\begin{align}\label{eq_grid1}
\{ (\delta k, \delta j): k,j \in \mathbb{Z} \}, \qquad \delta>0,
\end{align}
and possibly only approximately so due to numerical errors.
The goal of \emph{Gabor analysis} is to extract useful information about $f$ from such limited measurements. Equivalently, by \eqref{eq_stft}, the task is to capture the analytic function $F$ given a limited number of its samples on a grid. This second point of view led to the most conclusive results in Gabor theory, such as the complete description of all grids \eqref{eq_grid1} for which the Gabor transform fully retains the original analog signal $f$ \cite{MR924682, MR1188007, MR1173117, MR1173118}.

While Gabor signal analysis has traditionally focused on large values of the short-time Fourier transform \eqref{eq_stft}, recent work has brought to the foreground the rich information stored in its zeros, especially when the signal is contaminated with noise. Heuristically, the zeros of the Bargmann transform of noise exhibit a rather rigid random pattern with predictable statistics, from which the presence of a deterministic signal can be recognized as a salient local perturbation \cite{gardner2006sparse, 7180335, 7869100}. Remarkably, the Bargmann transform of white noise has been identified as a certain Gaussian analytic random function
\cite{MR4047541, MR4162314}, and consequently the well-researched statistics of their zero sets \cite{gafbook, MR2643444} can be leveraged in practice \cite[Chapters 13 and 15]{flandrin2018explorations}, \cite{MR4047541}. The particular structure observed in the zeros of the Bargmann transform under even a moderate amount of white noise has also been invoked as explanation for the sparsity resulting from certain non-linear procedures to sharpen spectrograms \cite{gardner2006sparse}, as the zeros of the Bargmann transform are repellers of the reassignment vector field \cite[Chapter 12]{flandrin2018explorations}. The practical exploitation of such insights requires an effective computational link between finitely given data on the one hand and zeros of Bargmann transforms of analog signals on the other.

\subsection{Computation of zero sets}
Suppose that the values of the Bargmann transform $F$ of a signal $f$ are given on a grid 
\begin{align}\label{eq_grid}
\grid = \{ \delta k + i \delta j: k,j \in \mathbb{Z} \}, \qquad \delta>0,
\end{align}
and we wish to compute an approximation of 
$\{F=0\}$, the zero set of $F$, within the square
\begin{align}\label{eq_domain}
\domain = \{x + i y: |x|, |y| \leq L\}.
\end{align}
More realistically, we only have access to samples of $F$ on those grid points near the computation domain, e.g., on
\begin{align}\label{eq_gridl}
\gridl = \{ \delta k + i \delta j: k,j \in \mathbb{Z}, |\delta k|, |\delta j| \leq L \}.
\end{align}
The inverse of the spacing of the grid, $1/\delta$, will be called the \emph{resolution} of the data.


\subsubsection*{Thresholding}
The most naive approach to compute $\{F=0\}$ is \emph{thresholding}: one selects all grid points $\gp$ such that $|F(\lambda)|$ is below a certain threshold $\varepsilon>0$:
\begin{align}\label{eq_threshold_test}
e^{-\frac12 |\gp|^2} |F(\lambda)| < \varepsilon.
\end{align}
The normalizing weight $e^{-\frac12|\gp|^2}$ is motivated by \eqref{eq_stft}, as the short-time Fourier transform of a typical signal can be expected to be bounded. One disadvantage of this approach is that it requires an educated choice for the threshold $\varepsilon$. Moreover, computations with various reasonable choices of thresholds, such as quantiles of $e^{-\frac12 |\gp|^2} |F(\lambda)|$ calculated over all grid points $\gp$, either fail to compute many of the zeros or capture too many points (see Figure \ref{fig_threshold}).

\begin{figure}[tb]
	\centering
    \includegraphics[width=\textwidth,trim={5cm 0 15cm 0},clip]{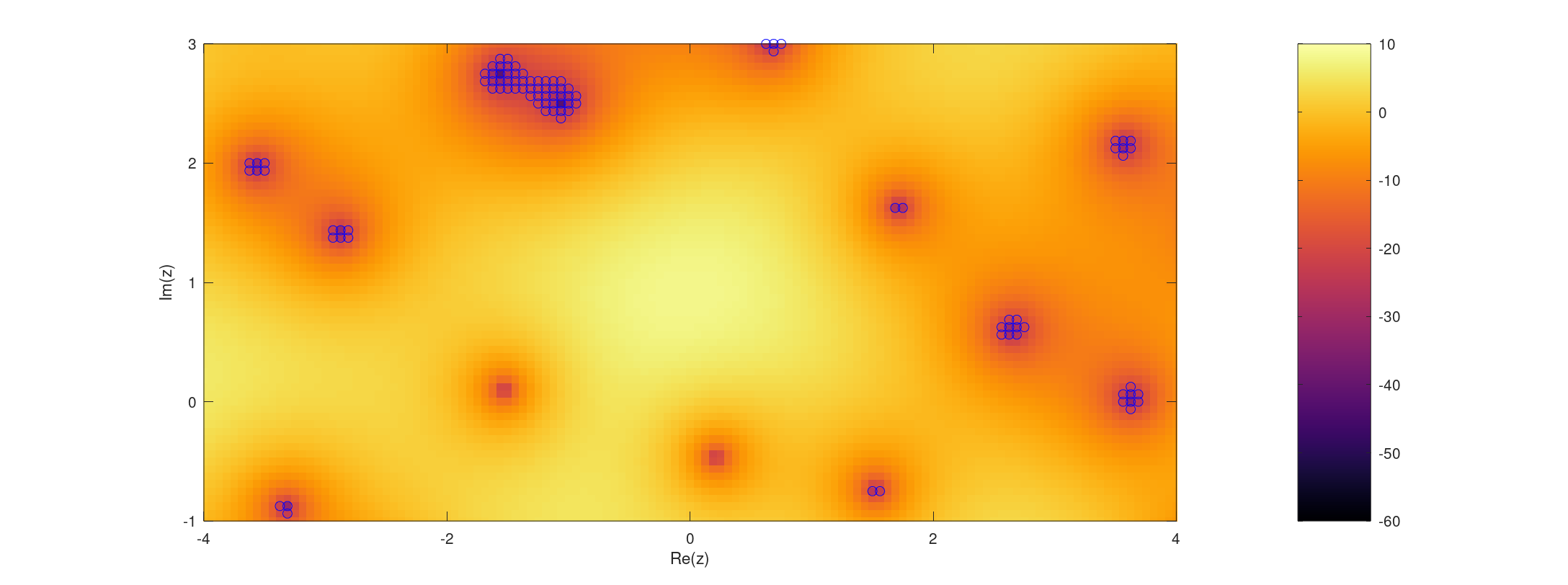}
	\caption{Calculation of zero sets by thresholding: Values below the same threshold (marked by circles) fail to detect some zeros and at the same time cannot clearly separate other zeros.}
	\label{fig_threshold}
\end{figure}

\subsubsection*{Extrapolation}
One may consider using the samples of $F$ on the finite grid \eqref{eq_gridl} to reconstruct the signal $f$, resample $F$ at arbitrarily high density, and thus calculate more easily the zero set $\{F=0\}$. However, computation of zeros through extrapolation may be inaccurate: while the samples of $F$ on the infinite grid \eqref{eq_grid}
determine $F$ as soon as $\sqrt{\pi} \cdot \delta < 1$ \cite{MR924682, MR1188007, MR1173118}, the truncation errors involved in the approximation of $F$ near $\domain$ from finite data \eqref{eq_gridl} can only be neglected at very high resolution $1/\delta$. Even if the values of $F$ are successfully extrapolated to a higher resolution grid, the remaining computation is still not trivial, as, for example, simple thresholding may fail even at high resolution (see 
Figure \ref{fig_comparison_methods_nds} and Section \ref{sec_nums}).

\subsubsection*{Minimal Grid Neighbors}
A greatly effective numerical recipe 
for the computation of zeros of the Bargmann transform
can be found in the code accompanying \cite{7180335} --- although not explicitly described in the text. A grid point $\gp$ is selected as a numerical approximation of a zero if
$e^{-\frac12 |\gp|^2} |F(\lambda)|$ is minimal among grid neighbors, i.e., 
\begin{align}\label{eq_flandrin}
e^{-\frac12 |\gp|^2} |F(\lambda)| \leq e^{-\frac12 |\gpp|^2} |F(\gpp)|, \qquad |\gp-\gpp|_\infty=\delta,
\end{align}
where $|z|_\infty = \max\{|x|,|y|\}$. The subset of points that pass the test furnish the computation of $\{F=0\}$.
 This method, which we call \emph{minimal grid neighbors} (MGN), performs impressively as long as the grid resolution is moderately high. Indeed, we understand that the method is behind the simulations in \cite[Chapter~15]{flandrin2018explorations} which quite faithfully reproduce the statistics of the zeros of the Bargmann transform of complex white noise (that are known analytically \cite{MR4047541, gafbook}). 
The MGN algorithm was also used to produce the plots in \cite{MR4047541}, as pointed out in \cite[Section~5.1.1.]{MR4047541};
see also \cite[Section~5]{MR4162314},
\cite[Section~IV]{koliander2019filtering}, and \cite{techrepsampta}. 
Heuristically, the test \eqref{eq_flandrin} succeeds in identifying zeros due to the analyticity of $F$, which implies that $|F(z)| e^{-\frac12 |z|^2}$ does not have non-zero local minima \cite[Section 8.2.2]{gafbook}. Remarkably, \eqref{eq_flandrin} is also effective even if the comparison involves only neighboring \emph{grid} points. 

The MGN algorithm performs equally well when calculating the zeros of the Bargmann transform of a signal
\begin{align}\label{eq_signal}
f = f_1 + \sigma \cdot \mathcal{N}
\end{align}
composed of a deterministic real-variable function $f_1$ plus complex white noise with variance $\sigma^2>0$. The presence of a certain amount of randomness must be behind the success of the algorithm, as, for $\sigma=0$, the method cannot be expected to succeed. Indeed, one can engineer a deterministic signal $f$ where the detection of zeros fails, as the value of its Bargmann transform $F$ can be freely prescribed on any given finite subset of the computation domain \cite{MR1173118}. We are unaware of performance guarantees for MGN.

\subsection{Contribution}
In this article we introduce a variant of MGN, called Adaptive Minimal Grid Neighbors (AMN). The algorithm is based on a comparison similar to \eqref{eq_flandrin} but incorporates \emph{an adaptive decision margin}, that depends on the particular realization of $F$. While AMN has the same mild computational complexity and similar practical effectiveness as MGN, we are able to estimate the accuracy and confidence of the computation with AMN in terms of the grid resolution.
In this way, we show that AMN is \emph{probably approximately correct} for the signal model \eqref{eq_signal}, in the sense that it computes the zero set with high probability up to the resolution of the data.

On the one hand, we present what to the best of our knowledge are the first formal guarantees for the approximate computation of zero sets of analytic functions from grid values. In fact, besides its main purpose of computation with specific data, the AMN algorithm offers a computationally attractive and provably correct method to \emph{simulate} zero sets of the Gaussian entire function \eqref{eq_gef}, by running the procedure with simulated inputs. On the other hand, our analysis is a first step towards understanding the performance of MGN.

\section{Main Result}
\subsection{The adaptive minimal neighbors algorithm}
We now introduce a new algorithm to compute zero sets of Bargmann transforms. Suppose again that samples of an analytic function $F \colon \mathbb{C} \to \mathbb{C}$ are given on those points of the grid \eqref{eq_grid} that are near the computation domain \eqref{eq_domain}, say, on
\begin{align}\label{eq_gridplus}
\gridlplus = \{ \delta k + i \delta j: k,j \in \mathbb{Z}, |\delta k|, |\delta j| \leq L+2\delta \}.
\end{align}
For each grid point $\gp$ strictly inside the computation domain, $\gp \in \gridl = \grid \cap \domain$, we use the neighboring sample at $\gp + \delta$ to compute the following \emph{comparison margin}:
\begin{align}\label{eq_th_intro}
\eta_\gp = 
e^{-\frac{1}{2} | \gp|^2} 
\max\big\{
\abs{F(\gp)}, 
\tfrac{3}{4}\bigabs{
e^{-\delta \bar{\gp}} F(\gp + \delta) - F(\gp)
}
\big\}.
\end{align}
The margin $\eta_\gp$ therefore depends on the particular realization of $F$. To motivate the definition, note that, when $\gp=0$,
the maximum in \eqref{eq_th_intro} is taken over
$|F(0)|$ and the absolute value of an incremental approximation of $\partial F(0)$, where $\partial F=\frac12(\frac{d}{dx} F-i \frac{d}{dy} F)$. The comparison margin thus incorporates the size and oscillation of $F$ at $z=0$. In general, $\eta_\gp$ has a similar interpretation with respect to the \emph{covariant derivative}
\begin{align}\label{eq_covder}
\bar{\partial}^* F(z) = \bar{z}\,F(z) - \partial F(z),
\end{align}
and, indeed, $\eta_\gp$ is defined so that
\begin{align}\label{eq_intro_approx}
\eta_\gp \approx e^{-\frac12|\gp|^2}
\max\big\{ |F(\gp)|,
\tfrac{3}{4} \abs{\bar{\partial}^* F(\lambda)} \delta \big\}.
\end{align}
The differential operator $\bar{\partial}^* F$ plays a distinguished role in the analysis of vanishing orders of Bargmann transforms \cite{bs93, esharo21} because it commutes with the translational symmetries of the space that they generate (the Bargmann-Fock shifts defined in Section \ref{sec_fock}).

The first step of the algorithm selects all grid points $\gp \in \gridl$ that pass the following comparison test:
\begin{align}\label{eq_test_intro}
e^{-\frac12 |\gpp|^2}
|F(\gpp)| \geq e^{-\frac12|\gp|^2} |F(\gp)| + \eta_{\gp}, \quad
\mbox{ whenever } |\gp-\gpp|_\infty = 2\delta, \quad \gpp \in \grid.
\end{align}
In contrast to \eqref{eq_flandrin}, the comparison in \eqref{eq_test_intro} does not involve the immediate grid neighbors of $\gp$ but rather those points lying on the square centered at $\gp$ with half-side-length $2\delta$; see Figure \ref{fig:diagram}. 
\begin{figure}[tbph]
	\centering
	\begin{subfigure}[b]{0.45\textwidth} 
    \centering\captionsetup{width=.8\linewidth}%
		\includegraphics[width=\textwidth,trim={2cm 0 6cm 0},clip]{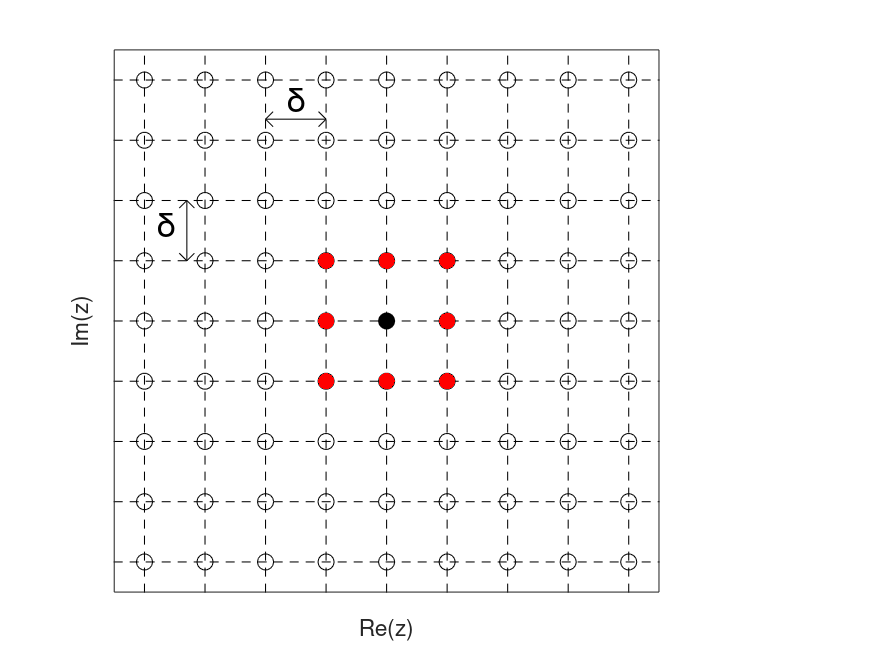}
		\caption{One of the immediate neighbors is used to calculate a comparison margin.}
		\label{fig:diagram_one_delta}
	\end{subfigure}%
	\begin{subfigure}[b]{0.45\textwidth}
		\centering\captionsetup{width=.8\linewidth}%
		\includegraphics[width=\textwidth,trim={2cm 0 6cm 0},clip]{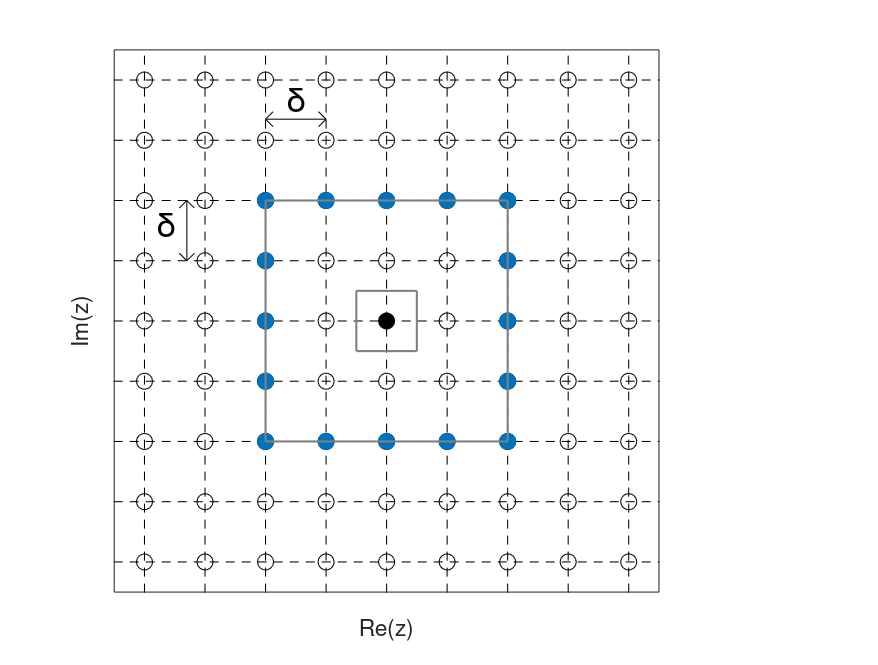}
		\caption{The weighted values of $F$ are compared against points in the larger box.}
		\label{fig:diagram_three_delta}
	\end{subfigure}
	\caption{The selection step of AMN.}
	\label{fig:diagram}
\end{figure}
(In particular, the test only involves grid points $\gpp \in \gridlplus$.) 
Intuitively, the larger 
distance between $\gp$ and $\gpp$ permits neglecting the error in the differential approximation
\eqref{eq_intro_approx}.

The use of non-immediate neighbors in \eqref{eq_test_intro} introduces a certain redundancy in the selection of numerical zeros, because the comparison boxes delimited by $\{\mu: |\gp-\gpp|_\infty = 2\delta\}$ overlap and, as a consequence, one zero of $F$ can trigger multiple positive tests; see Figure \ref{fig:before_sieving}. 
The second step of the algorithm sieves the selected points to enforce a minimal separation of $5\delta$ between different points. The algorithm is formally specified below.

\begin{figure}[tp]
	\centering
	\begin{subfigure}[t]{0.45\textwidth}    
    \centering\captionsetup{width=.8\linewidth}%
		\includegraphics[width=\textwidth,trim={2cm 0 6cm 0},clip]{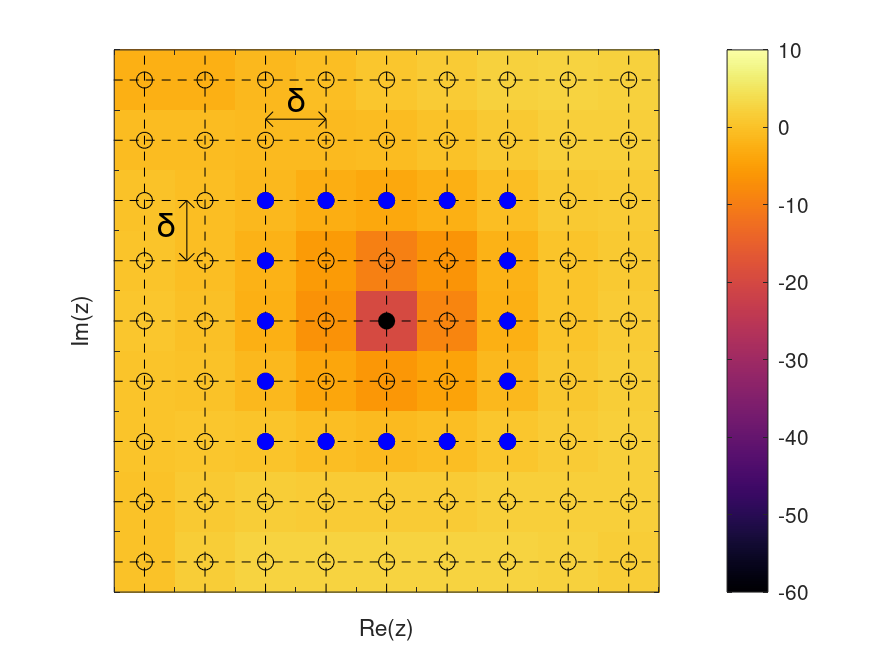}
		\caption{A point passes the selection test.}
		\label{fig:left_sieving}
	\end{subfigure}%
	\begin{subfigure}[t]{0.45\textwidth}
		\centering\captionsetup{width=.8\linewidth}%
		\includegraphics[width=\textwidth,trim={2cm 0 6cm 0},clip]{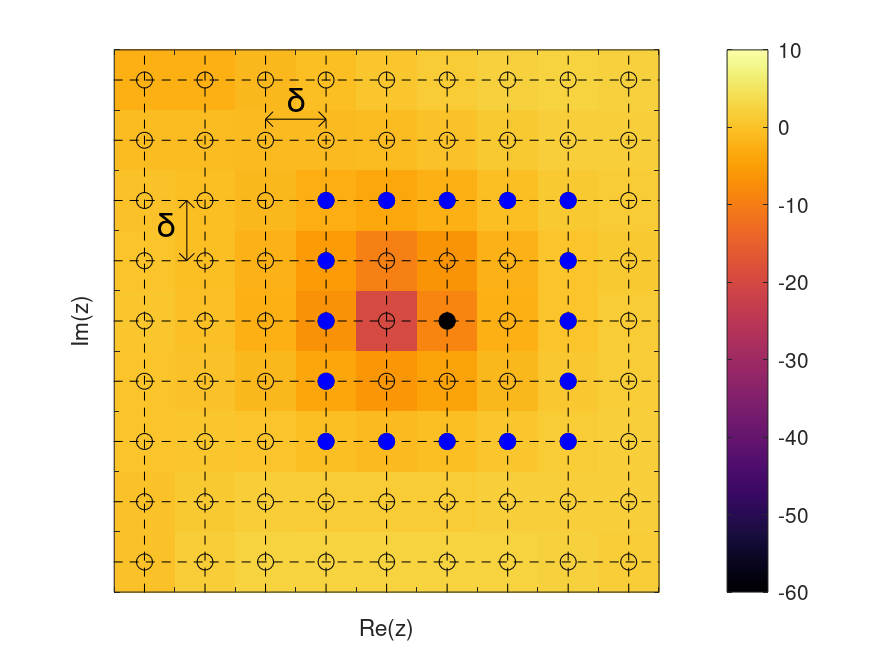}
		\caption{The same zero causes a second detection.}
		\label{fig:right_sieving}
	\end{subfigure}
	\caption{Sliding the test box.}
	\label{fig:before_sieving}
\end{figure}

\noindent\rule{\textwidth}{1pt}
\noindent {\bf Algorithm AMN: \,}{Compute zero set of $F$ inside $\domain = [-L,L]^2$}\vspace{-0.2cm}\\
\noindent\rule{\textwidth}{1pt}

\alstep{Input}: A domain length $L\geq 1$, a grid spacing parameter $\delta >0$ and samples of a function $F\colon \mathbb{C} \to \mathbb{C}$ on the grid points $\gridlplus = \grid \cap [-(L+2\delta),L+2\delta]^2$. 

\medskip

\alstep{Selection step}: For each grid point $\gp \in \gridl = \grid \cap \domain$ inside the target domain $\domain$, we define the following \emph{comparison margin}:
\begin{align}\label{eq_th}
\eta_\gp = 
e^{-\frac{1}{2} | \gp|^2} 
\max\big\{
\abs{F(\gp)}, 
\tfrac{3}{4}
\big|
e^{-\delta \bar{\gp}} F(\gp + \delta) - F(\gp)
\big|
\big\}.
\end{align}

\noindent The grid point $\gp$ is then \emph{selected} if the following test is satisfied:
\begin{align}\label{eq_test}
e^{-\frac12 |\gpp|^2}
|F(\gpp)| \geq e^{-\frac12|\gp|^2} |F(\gp)| + \eta_{\gp}, \quad
\mbox{ whenever } |\gp-\gpp|_\infty = 2\delta, \quad \gpp \in \grid.
\end{align}
(Note that the test \eqref{eq_test} only involves grid points $\gpp \in \gridlplus$.)

\noindent Let $\outseta$ be the set of all selected grid points.

\medskip

\alstep{Sieving step}: Use an off-the-shelf clustering algorithm to select a subset $\outset \subset \outseta$ that is $5\delta$ separated:
\begin{align}\label{eq_7s}
\inf \big\{ |\gp-\gpp|_\infty : \gp, \gpp \in \outset, \gp \not= \gpp \big\} \geq 5 \delta
\end{align}
and maximal with respect to this property, i.e., no proper superset satisfies \eqref{eq_7s}.

\noindent (One concrete implementation of the sieving step is described in Section \ref{sec_siev_step}.)

\medskip

\alstep{Output}: The set $\outset$.\\
\noindent\rule{\textwidth}{1pt}

\smallskip

\begin{rem}\label{rem_cons}
The constants $3/4$ in \eqref{eq_th}, $2$ in \eqref{eq_test}, and $5$ in \eqref{eq_7s} are to some extent arbitrary, and other choices lead to similar results. These particular values are chosen to aid the exposition rather than to optimize practical performance. In fact, these choices are suboptimal at low resolutions (see Section \ref{sec_nums}).
\end{rem}

\subsection{Performance guarantees for AMN}
To study the performance of the AMN algorithm we introduce the following \emph{input model}, which, as we will explain, corresponds to the Bargmann transform of an arbitrary signal contaminated with complex white noise of an arbitrary intensity.

\subsubsection*{Input model}
We consider a random entire function on the complex plane 
\begin{align}\label{eq_rf}
F = F^1 + \noise \cdot F^0,
\end{align}
where $F^1\colon \mathbb{C} \to \mathbb{C}$ is a deterministic entire function, $F^0$ is a (zero mean) Gaussian analytic function with correlation kernel:
\begin{align}\label{eq_gef}
\mathbb{E} \big\{ F^0(z) \cdot \overline {F^0(w)} \big\}=e^{z \bar{w}},
\qquad z,w \in \mathbb{C},
\end{align}
and $\noise>0$ is the \emph{noise level}. 
We assume that the deterministic function $F^1$ satisfies the quadratic exponential growth estimate
\begin{align}\label{eq_cf}
|F^1(z)| \leq \conf \cdot e^{\frac12|z|^2}, \qquad z \in \mathbb{C},
\end{align}
for some constant $\conf\geq 0$. 

As for $F^0$, the assumption means that for each $z_1, \ldots, z_n \in \mathbb{C}$, $(F^0(z_1), \allowbreak \ldots, \allowbreak F^0(z_n))$ is a normally distributed (circularly symmetric) complex random vector,
with mean zero and covariance matrix $\big[e^{z_k \overline{z_\ell}}\big]_{k,\ell}$.
Alternatively, $F^0$ can be described as
\begin{align}\label{eq_alt_gef}
F^0(z) = \sum_{n \geq 0} \tfrac{\xi_n}{\sqrt{n!}} z^n,
\end{align}
where $(\xi_n)_{n \geq 0}$ are independent standard complex random variables \cite{gafbook}.

\subsubsection*{Discussion of the model}
The random function $F^0$ is the Bargmann transform of standard complex white noise $\mathcal{N}$. As each realization of complex white noise is a tempered distribution, the computation of its Bargmann transform \eqref{eq_bar} is also to be understood in the sense of distributions, as in \cite{MR201959}. (See \cite{MR4162314} and \cite[Section~5.1]{gwhf} for a detailed discussion on this, and alternative approaches.)

We can similarly interpret $F^1$ as the Bargmann transform of a distribution $f^1$ on the real line \cite{MR201959}. The assumption \eqref{eq_cf} means precisely that $f^1$ belongs to the \emph{modulation space} $M^\infty(\mathbb{R})$ consisting of distributions with Bargmann transforms bounded with respect to the standard Gaussian weight --- or, equivalently, with bounded short-time Fourier transforms \cite{benyimodulation}. 
The modulation space $M^\infty(\mathbb{R})$ includes all square-integrable functions $f^1 \in L^2(\mathbb{R})$ and also many of the standard distributions used in signal processing.

In summary, the input model \eqref{eq_rf} corresponds exactly to the Bargmann transform of a random signal
\begin{align}\label{eq_signal_2}
f = f^1 + \sigma \cdot \mathcal{N},
\end{align}
where $f^1 \in M^\infty(\mathbb{R})$ and $\sigma \cdot \mathcal{N}$ is
complex white noise with standard deviation $\sigma$.

\subsubsection*{Performance analysis}
We now present the following performance guarantees, pertaining to the computation domain \eqref{eq_domain} and the acquisition grid \eqref{eq_gridplus}. To avoid immaterial technicalities, we assume that the corners of the computation domain lie on the acquisition grid.

\begin{theorem}\label{mth}
	Fix a domain width $L \geq 1$, a noise level $\sigma>0$, and a grid spacing $\delta >0$
	such that $L/\delta \in \mathbb{N}$.
	Let a realization of a random function $F$ as in \eqref{eq_rf} with \eqref{eq_gef} and \eqref{eq_cf} be observed on $\gridlplus$, and let $\outset$ be the output of the AMN algorithm.
	
	There exists an absolute constant $C$ such that, with probability at least
	\begin{align}\label{eq_suc_pro}
	1-C \, L^2  \exp\bigg(\frac{\conf^2}{8\sigma^2}\bigg)
\max\left\{1,\log^2(1/\delta)\right\}
\delta^{4},
	\end{align}
	there is an injective map $\map\colon \{F=0\} \cap \domain \to \outset$ with the following properties:
	
	$\bullet$ \emph{(Each zero is mapped into a near-by numerical zero)}
	\begin{align}\label{eq_dist}
	|\map(\zeta) - \zeta|_ \infty \leq 2 \delta, \qquad \zeta \in \{F=0\} \cap \domain.
	\end{align}
	
	$\bullet$ \emph{(Each numerical zero that is away from the boundary arises in this form)}
	
	For each $\gp \in \outset \cap \domainminus$ there exists
	$\zeta \in \{F=0\} \cap \domain$ such that $\gp=\map(\zeta)$.	
\end{theorem}
A proof of Theorem \ref{mth} is presented in Section \ref{sec_proof}.
We remark some aspects of the result.

\begin{itemize}[itemindent=0cm, leftmargin=0.3cm, itemsep=0.3cm]
	\item The AMN algorithm does not require knowledge of the noise level $\sigma$ and is homogeneous in the sense that $F$ and $c F$, with $c \in \mathbb{C} \setminus \{0\}$ produce the same output.
	
	\item Within the estimated success probability, the computation is accurate \emph{up to a factor of the grid spacing}.
	
	\item The analysis concerns an \emph{arbitrary deterministic signal} impacted by noise and is uniform over the class \eqref{eq_cf}. As usual in such \emph{smoothed analysis}, the success probability grows as the signal to noise ratio $\conf/\sigma$ decreases,
	because randomness helps preclude the very untypical features that could cause the algorithm to fail~\cite{spielman2009smoothed}. In fact, in the noiseless limit $\sigma=0$, the algorithm could completely fail, since $F^1$ can be freely prescribed on any finite subset of the plane \cite{MR1173118}. For example, irrespectively of its values on the acquisition grid, the deterministic function $F^1$ could have a cluster of zeros
	of small diameter that would trigger a single positive minimality test. The proof of Theorem \ref{mth} shows  that such examples are fragile, as the addition of even a moderate amount of noise regularizes the geometry of the zero set.
	
	\item Up to a small boundary effect, the guarantees in Theorem \ref{mth} comprise an estimate on the Wasserstein distance between the atomic measures supported on $\{F=0\} \cap \domain$ and on the computed set $\outset$.  More precisely, for a \emph{tolerance} $\tol >0$ let us define the \emph{boundary-corrected Wasserstein pseudo-distance} between two sets $U,V \subseteq \mathbb{C}$ as
	\begin{align*}
	\was_{L,\tol}(U,V) = \inf_\Phi \max_{z \in U} | \Phi(z)-z|_\infty,
	\end{align*}
	where the infimum is taken over all injective maps $\Phi\colon U \to V$ such that $V \cap \domaintol \subseteq \Phi(U)$. (The definition is not symmetric in $U$ and $V$, but this is not important for our purpose.) Then Theorem \ref{mth} reads
	\begin{align*}
	\mathbb{P} \Big[ \was_{L,2\delta}\left(\{F=0\} \cap \domain,\outset\right) > 2\delta \Big] \leq C L^2 \exp\bigg(\frac{\conf^2}{8\sigma^2}\bigg)
\max\left\{1,\log^2(1/\delta)\right\}
\delta^{4}.
	\end{align*}
	
	\item The presented analysis concerns a signal contaminated with \emph{complex} white noise. This is a mathematical simplification; we believe that with more technical arguments a similar result can be derived for \emph{real} white noise. The case of colored noise seems more challenging and will be the object of future work.
\end{itemize}

\subsection{Numerical experiments}

In Section \ref{sec_nums}, we report on numerical experiments that compare the AMN and MGN algorithms. We also include a modified version of thresholding (ST), that uses a thereshold proportional to the grid spacing and incorporates a sieving step as in AMN (while standard thresholding without sieving performs extremely poorly, as seen in Figure \ref{fig_threshold}).

The performance of AMN, MGN, and ST is first tested indirectly, by using these algorithms to simulate the zero sets of the random functions in the input model \eqref{eq_rf}. We then compare theoretically derived statistics of the zeros of \eqref{eq_rf} to empirical statistics obtained from the output of AMN and MGN under various simulated realizations of \eqref{eq_rf}.

Second, we perform a consistency experiment that aims at estimating the probability of computing a low-distortion parametrization of the zero set of $F$, as in Theorem \ref{mth}. Specifically, we simulate a realization of the random input $F$ sampled at high-resolution and use the output of AMN or MGN as a proxy for the ground truth $\{F=0\}$. We then test the extent to which this set is captured by the output of AMN, MGN, or ST from lower resolution subsets of the same simulated data.

The performance of AMN and MGN is almost identical, although the minimal resolution at which MGN starts to perform well is slightly lower than that for AMN. (This is to be expected, as the constants $2$ and $5$ used in \eqref{eq_test} and \eqref{eq_7s} are not adequate for low resolutions, cf.\ Remark \ref{rem_cons}.) Both AMN and MGN significantly outperform ST. See also
Figure \ref{fig_comparison_methods_nds} for an illustration.

\begin{figure}[tbp]
	\begin{minipage}{\textwidth}
		\begin{center}
			\begin{tikzpicture} 
				\draw (-0.5,-0.3) rectangle (5,0.3); 
				\draw (-0.2,0) circle (3.5pt); 
				\draw (1.5,0) node{AMN and MGN}; 
				\draw (4,0) node[cross=3.5pt]{};
				\draw (4.5,0) node{ST}; 
			\end{tikzpicture}
		\end{center}
	\end{minipage}	
	\begin{minipage}{\textwidth}
		\centering
		\begin{subfigure}[t]{0.49\textwidth}
			\centering
			\includegraphics[width=\textwidth,trim={1.5cm 0 3cm 0},clip]{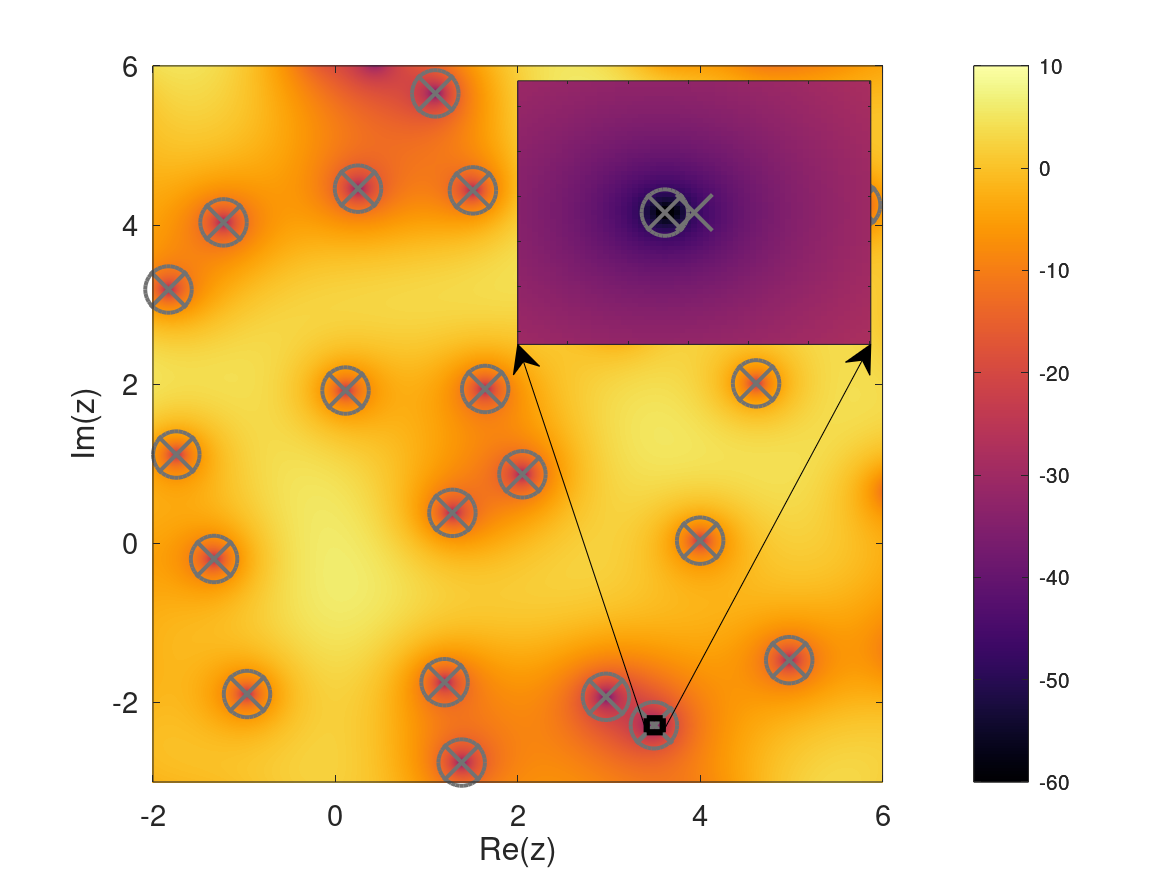}
			\caption{$f^1(t)=\exp(-t^2)$, $\conf=1$}    
			\label{fig_gauss_100_7_6_A_1}
		\end{subfigure}
		\hfill
		\begin{subfigure}[t]{0.49\textwidth}   
			\centering 
			\includegraphics[width=\textwidth,trim={1.5cm 0 3cm 0},clip]{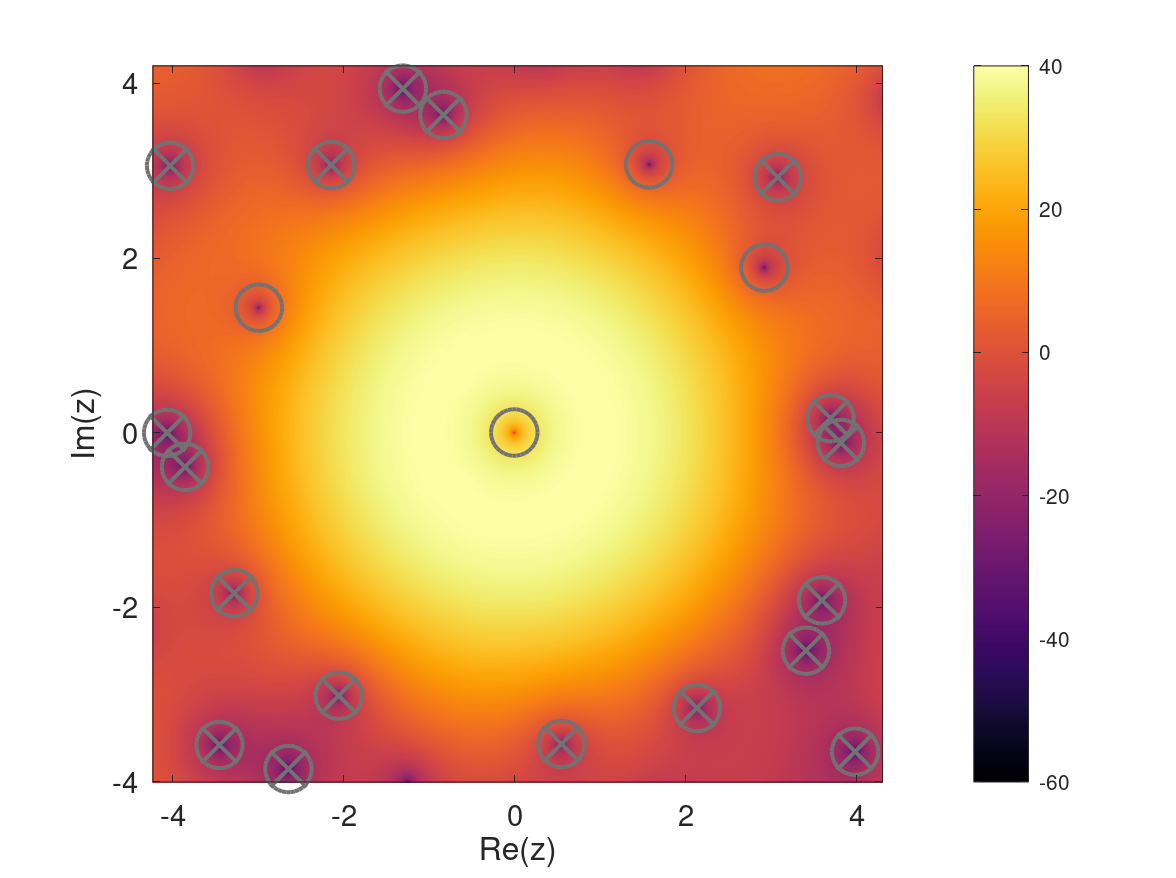}
			\caption{$f^1(t)=t \exp(-t^2)$, $\conf=100$}    
			\label{fig_herm1_100_7_6_A_100}
		\end{subfigure}
	\end{minipage}
	\caption{\small A realization of $e^{-\frac12 |z|^2} |(F^0(z) + F^1(z))|$ with $F^1$ the Bargmann transform of $f^1$. The deterministic functions are scaled to obtain the prescribed $\conf$. Zeros computed with AMN, MGN, and ST are calculated from grid samples with $\delta=2^{-9}$. Zeros from AMN and MGN coincide (circle), while ST (cross) fails either by detecting false zeros (left)
		or by not capturing all of them (right).
	}%
	\label{fig_comparison_methods_nds}
\end{figure}
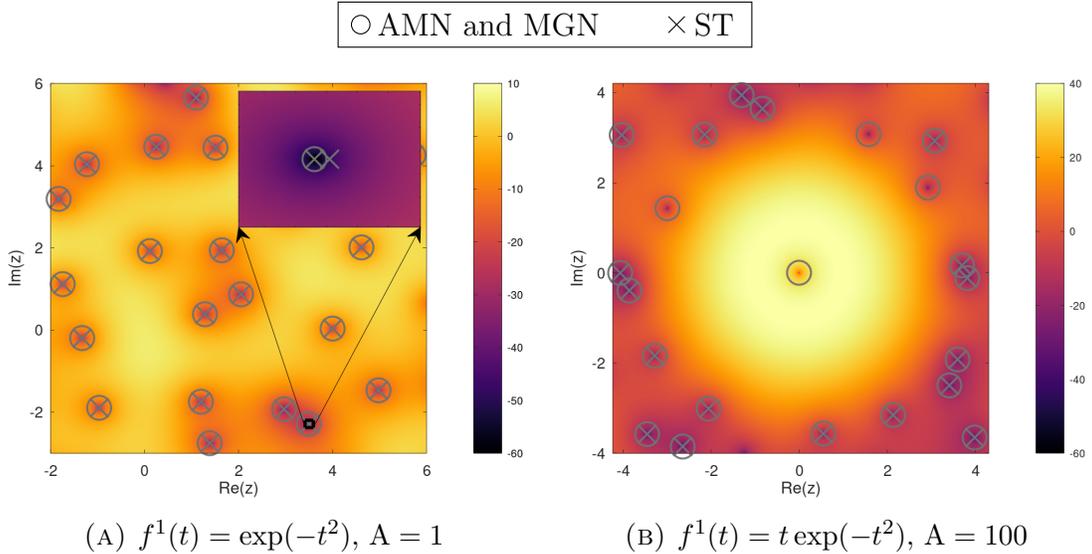

The favorable performance of AMN is interesting also when the input is just noise, as it gives a fast and provably accurate method to simulate the zeros of the Gaussian entire function \eqref{eq_gef}. (The simulations that we present 
in Section~\ref{sec_nums} use certain heuristic shortcuts to accelerate the simulation of the input \eqref{eq_gef} --- see Section~\ref{sec_sim}; although we do not formally analyze these, they are implicitly validated, as the simulated point process reproduces the expected theoretical statistics.)

All numerical experiments can be reproduced with openly-accessible software and our code is available at \url{https://github.com/laescudero/discretezeros}.
Our implementation of the Bargmann transform uses \cite{basw91}.

\subsection{Organization} Section \ref{sec_pre} introduces the notation and basic technical tools about analytic functions, Bargmann-Fock shifts, and their applications to random functions and their zeros. Theorem \ref{mth} is proved in Section \ref{sec_proof}, while numerical experiments are presented in detail in Section \ref{sec_nums}. Conclusions and outlook on future directions are discussed in Section \ref{sec_con}.

\section{Preliminaries}\label{sec_pre}
\subsection{Notation}\label{sec_not}
For a complex number $z=x+iy$, we use the notation $|z|_\infty = \max\{|x|,|y|\}$, while $|z|$ denotes the usual absolute value. The zero set of $F$ is denoted by $\{F=0\}$.
The differential of the (Lebesgue) area measure on the plane will be denoted for short $\dm$, while the measure of a set $E$ is $|E|$.
With a slight abuse of notation, we also denote the cardinality of a finite set $Z$ by $\card{Z}$.
Squares on the complex plane are denoted by $Q_r(z) = \{w \in \mathbb{C}: |z-w|_\infty \leq r\}$. 
For two non-negative functions $f,g$, we write
$f \lesssim g$ if there exists an \emph{absolute constant} $C$ such that $f(x) \leq C g(x)$, for all $x$. We write $f \asymp g$ if $f \lesssim g$ and $g \lesssim f$.

The Wirtinger derivative of a function
$F\colon \mathbb{C} \to \mathbb{C}$ is $\partial F=\frac12(\frac{d}{dx} F-i \frac{d}{dy} F)$. When we need to stress on which variable the derivative is taken we write subindices, e.g., $\partial_w F(z,w)$.

A Gaussian entire function (see \cite[Ch.~2]{gafbook} and \cite{MR2643444}) is a random function $F\colon \mathbb{C} \to \mathbb{C}$ that is almost surely entire, and such that for every $z_1, \ldots, z_n \in \mathbb{C}$, $\big(F(z_1), \ldots, F(z_n)\big)$ is a circularly symmetric complex normal vector. 
We will be only concerned with the random function $F$ given in \eqref{eq_rf}. 
We also use the notation \eqref{eq_grid}, \eqref{eq_domain}, \eqref{eq_gridl}, possibly for distinct values of $L$.

\subsection{Bargmann-Fock shifts and stationarity of amplitudes}\label{sec_fock}
The analysis of the AMN algorithm is more transparent when formulated in terms of the \emph{Bargmann-Fock shifts}. For a function $F\colon \mathbb{C} \to \mathbb{C}$ we let
\begin{align}\label{eq_fs}
F_w(z) = e^{-\frac{1}{2} |w|^2 - z \overline{w}} \, F(z+w).
\end{align}
The \emph{amplitude} of an entire function $F$ is defined as the weighted magnitude
\begin{align}\label{eq_G}
G(z) = e^{-\frac12 |z|^2} |F(z)|,
\end{align}
and satisfies 
\begin{align}\label{eq_sat}
G_w(z) :=
G(z+w)  = e^{-\frac12 |z|^2} |F_w(z)|.
\end{align}
The comparison margin of the AMN algorithm \eqref{eq_th} can be expressed in terms of Bargmann-Fock shifts as
\begin{align}\label{eq_thb}
\eta_\gp 
= 
\max\big\{
\abs{F_\gp(0)}, 
\tfrac{3}{4}
\big|
 F_\gp(\delta) - F_\gp(0)
\big|
\big\},
\end{align}
and leads to the approximation \eqref{eq_intro_approx} because
\begin{align*}
|F'_\gp(0)| = e^{-\frac12 |\gp|^2} \abs{\bar{\partial}^* F(\lambda)}.
\end{align*}
(Here and throughout we write $F'_\gp(0)$ for $\partial[F_\gp](0)$.)
Similarly, in terms of amplitudes, the test \eqref{eq_test} reads
\begin{align}\label{eq_testb}
G(\gpp)=
e^{-\frac12|\gpp-\gp|^2}|F_\gp(\gpp-\gp)|  \geq |F_\gp(0)| + \eta_\gp
= G(\gp) + \eta_\gp,
\quad
\gpp \in \grid, |\gpp-\gp|_\infty = 2 \delta.
\end{align}

With respect to the input model \eqref{eq_rf} we note that, if $F^0$ is the (zero mean) Gaussian entire function with correlation kernel \eqref{eq_gef}, then the Bargmann-Fock shifts $F^0 \mapsto F^0_w$ preserve the stochastics of $F^0$, as they leave its covariance kernel invariant. As a consequence, for any $w \in \mathbb{C}$,
$F^0_w(0), \big[F^{0}_{w}\big]'(0)$ are independent standard complex normal random variables (with zero mean and variance $1$). Indeed, by the mentioned invariance it suffices to consider $w=0$, and, in this case, $F^0_w(0), \big[F^{0}_{w}\big]'(0)$ are the coefficients $\xi_0$ and $\xi_1$ in \eqref{eq_alt_gef}.

\subsection{Minimum principle for amplitudes}
The following weight\-ed version of the minimum principle is at the core of the success of MGN and AMN.
\begin{lemma}\label{lemma_min}
	Let $F\colon  \mathbb{C} \to \mathbb{C}$ be entire, $r>0$, and assume that
	\begin{align}\label{eq_testa}
	|F(0)| \leq |F(z)| e^{-\frac12 |z|^2},
	\quad
	\mbox{for all } z \in \mathbb{C} \mbox{ such that } |z|_\infty = r.
	\end{align}
	Then there exists $z \in \mathbb{C}$ with $|z|_\infty \leq r$ such that $F(z)=0$.
\end{lemma}
\begin{proof}
	Let $D := \{ z \in \mathbb{C}: |z|_\infty < r\}$ and suppose that $F$ does not vanish on  $\bar{D}$. Then
	the function
	\begin{align*}
	H(z) = \frac{e^{\frac12 |z|^2}}{|F(z)|}
	\end{align*}
	is well defined on $\bar{D}$ and satisfies
	\begin{align}\label{eq_maxH}
	\log H (0) \geq \log H (z), \qquad z \in \partial D.
	\end{align}
	By the analyticity of $F$, $\Delta \log|F| = 0$ and thus
	\begin{align*}
	\Delta \big[ \log H (z)\big] = \Delta \Big[
	\tfrac{|z|^2}{2} - \log|F(z)|
	\Big] = 2.
	\end{align*}
	Hence, the maximum principle for subharmonic functions together with \eqref{eq_maxH} implies that $\log H(z)$ and therefore $|F(z)| e^{-\frac12 |z|^2}$ is constant on $D$.
	For $z \in D$, we compute
	\begin{align*}
	0 &= \partial_z \big[|F(z)|^2 e^{-|z|^2}\big] = 
	\partial_z \big[F(z) \,e^{-|z|^2} \big]  \overline{F(z)}
	\\&=\big[\partial_z F(z)-\bar{z}F(z)\big] \overline{F(z)} \,e^{-|z|^2}.
	\end{align*}
	As $F$ is non-vanishing on $D$, it follows that $\partial_z F-\bar{z}F = 0$ on $D$, and therefore
	\begin{align*}
	0 = \partial_{\bar{z}} [\partial_z F-\bar{z}F]
	= - F,
	\end{align*}
	on $D$. This contradiction shows that $F$ must vanish on $\bar{D}$.
\end{proof}

\subsection{Linearization}
In what follows, we derive basic facts about the input model \eqref{eq_rf}, and always assume that \eqref{eq_gef} and \eqref{eq_cf} hold.

The following is a strengthened version of \cite[Lemma 2.4.4]{gafbook}.
\begin{lemma}\label{lemma_lin}
	Let $F$ be as in \eqref{eq_rf}. Then there exists an absolute constant $C>0$ such that for all $L \geq 1$ and $t\geq\conf$,
	\begin{align*}
	&\mathbb{P} \left[ \sup_{w \in \domain, |z| \leq 10} |z|^{-2}
	\big| F_w(z) - \big(F_w(0) + F_w'(0) z \big) \big| > t \right] \leq CL^2 e^{-(t-\conf)^2/(8\noise^2)},
	\\
	&\mathbb{P} \left[ \sup_{w \in \domain, |z| \leq 10}
	|z|^{-2} \big| F_w(z)e^{-\frac12 |z|^2} - \big(F_w(0) + F_w'(0) z \big) \big|
	> t \right] \leq CL^2 e^{-(t-\conf)^2/(8\noise^2)}.
	\end{align*}
\end{lemma}
\begin{proof}	
	We consider the Taylor expansion of $F$:
	\begin{align*}
	F(z) = F(0) + F'(0) z + E_2(z) z^2,
	\end{align*}
	where $E_2$ can be bounded in terms of the amplitude \eqref{eq_G} as
  	\begin{equation*}
	\sup_{|z| \leq 10} |E_2(z)| 
  \lesssim \int_{|\zeta| \leq 12} |F(\zeta)| \,\dm(\zeta) 
  \asymp \int_{|\zeta| \leq 12} |G(\zeta)| \,\dm(\zeta).
	\end{equation*}
	We also note that for $|z| \leq 10$, 
	\begin{align*}
	|F(z) - F(z) e^{-\frac12 |z|^2}| &= |F(z)| \big| 1-e^{-\frac12 |z|^2} \big|
	\\
	&\lesssim
	|z|^2 \int_{|\zeta| \leq 12} |F(\zeta)| \,\dm(\zeta) 
	\asymp |z|^2 \int_{|\zeta| \leq 12} |G(\zeta)| \,\dm(\zeta).
	\end{align*}
	Hence, for $|z| \leq 10$,
	\begin{align*}
	\big| F(z) - \big(F(0) + F'(0) z \big) \big| \lesssim
	|z|^2 \int_{|\zeta| \leq 12} |G(\zeta)| \,\dm(\zeta),
	\\
	\big| F(z)e^{-\frac12 |z|^2} - \big(F(0) + F'(0) z \big) \big| \lesssim
	|z|^2 \int_{|\zeta| \leq 12} |G(\zeta)| \,\dm(\zeta).
	\end{align*}
	
	We apply the previous bounds to $F_w$, note that,
	by \eqref{eq_sat}, $G_w(z) = G(z+w)$, and obtain that for $|z| \leq 10$,
	\begin{align*}
	A_w(z) &:= |z|^{-2}
	\big| F_w(z) - \big(F_w(0) + F_w'(0) z \big) \big| \lesssim
	\int_{|\zeta-w| \leq 12} |G(\zeta)| \,\dm(\zeta),
	\\
	B_w(z) &:= |z|^{-2} \big| F_w(z)e^{-\frac12 |z|^2} - \big(F_w(0) + F_w'(0) z \big) \big| \lesssim
	\int_{|\zeta-w| \leq 12} |G(\zeta)| \,\dm(\zeta).
	\end{align*}
	Hence
	\begin{align*}
	\sup_{|z| \leq 10, w \in \domain} A_w(z) + B_w(z)
	\lesssim \sup_{w\in \domain} \int_{|\zeta-w| \leq 12} |G(\zeta)| \,\dm(\zeta)
	\lesssim \sup_{|\zeta| \leq L+12} |G(\zeta)|.
	\end{align*}
	Let $G^0$ and $G^1$ be the amplitudes corresponding to $F^0$ and $F^1$, respectively. Then by \eqref{eq_cf},
	\begin{align*}
	|G(\zeta)| \leq \noise \cdot \abs{G^0(\zeta)} + |G^1(\zeta)| \leq \conf + \noise \cdot |G^0(\zeta)|,
	\qquad \zeta \in \mathbb{C}.
	\end{align*}
	Hence,
	\begin{equation*}
	\mathbb{P} \Big[ \sup_{|\zeta| \leq L} |G(\zeta)| > t \Big] \leq \mathbb{P} \Big[ \sup_{|\zeta| \leq L} \abs{G^0(\zeta)} > \frac{t-\conf}{\noise} \Big].
	\end{equation*}
	To conclude, we claim that the following excursion bound holds:
	\begin{align*}
	\mathbb{P} \Big[ \sup_{|\zeta| \leq L} |G^0(\zeta)| > t \Big] \leq C L^2 e^{-t^2/8}, \qquad t \geq 0,
	\end{align*}
	where $C>0$ is an absolute constant. For $L \leq 1/4$ this follows for example from \cite[Lemma 2.4.4]{gafbook}. In general, we cover the domain with $\lesssim L^2$ squares of the form $w+[-1/4,1/4]^2$, apply the previously mentioned bound to $G^0_w(z)=G^0(z+w)$, and use a union bound. This completes the proof.
\end{proof}

\subsection{Almost multiple zeros}
It is easy to see that, almost surely, the random function \eqref{eq_rf} has no multiple zeros. 
In the analysis of the AMN algorithm, we will also need to control the occurrence of zeros that are multiple up to a certain numerical precision, in the sense that $F$ and its derivative are simultaneously small. The following lemma is a first step in that direction, as it controls the probability of finding a grid point that is an almost multiple zero.

\begin{lemma}\label{lemma_mult}
	Let $F$ be as in \eqref{eq_rf} and $\alpha, \beta>0$. Then the probability that for some grid point $\gp \in \gridl$ the following occurs:
	\begin{align}\label{eq_a}
	|F_\gp(0)| \leq \alpha, \text{ and }\ |F'_\gp(0)| \leq \beta
	\end{align}
	is at most $C L^2 \alpha^2 \beta^2 \delta^{-2} \noise^{-4}$, where $C$ is an absolute constant.
\end{lemma}
\begin{proof}
	For each grid point $\gp \in \gridl$, $F_\gp(0)$ and $F'_\gp(0)$ are independent complex normal variables with possibly non-zero means $\mu_1, \mu_2$ and variance $\noise^2$.
	Therefore, 
	\begin{align}\label{eq_a1}
	P\big(|F_\gp(0)|\leq \alpha\big) = 
	\frac{1}{\pi \noise^2}
	\int_{\abs{\zeta}\leq \alpha} e^{-\frac{1}{\noise^2} \abs{\zeta-\mu_1}^2}\, \dm(\zeta),
	\\
	\label{eq_a2}
	P\big(|F'_\gp(0)|\leq \beta\big) = 
	\frac{1}{\pi \noise^2}
	\int_{\abs{\zeta}\leq \beta} e^{-\frac{1}{\noise^2} \abs{\zeta-\mu_2}^2}\, \dm(\zeta).
	\end{align}
	By Anderson's lemma \cite{MR69229}, the right-hand sides of \eqref{eq_a1} and \eqref{eq_a2} are maximal when $\mu_1=0$ and $\mu_2=0$, respectively. Direct computation in those cases yields
	$P(|F_\gp(0)|\leq \alpha) \lesssim \alpha^2 \noise^{-2}$ and 
	$P(|F'_\gp(0)|\leq \beta) \lesssim \beta^2 \noise^{-2}$. By independence, the probability of \eqref{eq_a} is $\lesssim \alpha^2 \beta^2 \noise^{-4}$. On the other hand, there are $\lesssim L^2 \delta^{-2}$ grid points under consideration, so the conclusion follows from the union bound.
\end{proof}

\subsection{First intensity of zeros}
\label{sec_fifornzmean}
The following proposition is not used in the proof of Theorem \ref{mth}, but rather as a benchmark in the numerical experiments (Section~\ref{sec_nums}).
\begin{prop}
	\label{prop_first}
	Let $F$ be as in \eqref{eq_rf}. Then for every Borel set $B \subset \mathbb{C}$,
	\begin{align*}
	\E[\card{\{z\in B: F(z)=0\}}] = \int_B \rho_1(\zeta) \,\dm(\zeta)
	\end{align*}
	where
	\begin{equation}\label{eq_rho1}
	\rho_1(\zeta) = \frac{1}{\pi} e^{-\frac{1}{\noise^2}\abs{F^1(\zeta)}^2 e^{-\abs{\zeta}^2}} \left(1+\frac{e^{-\abs{\zeta}^2}}{\noise^2} \left|\partial_\zeta {F^1}(\zeta)-\overline{\zeta} F^1(\zeta)\right|^2\right).
	\end{equation}
\end{prop}
\begin{proof}
  The set of zeros and thus $\rho_1(\zeta)$ does not change if we scale $F$ by a fixed constant. 
  Hence, by considering the function $\frac{1}{\noise}F$ in place of $F$, we can assume that $\noise=1$.
	The expected number of points $\{z\in B: F(z)=0\}$ of a Gaussian random field $F$ is given by Kac-Rice's formula:
	\begin{equation}
	\E[\card{\{z\in B: F(z)=0\}}]
	= \int_B  \E\big[\abs{\det DF(\zeta)} \, \big\vert \, F(\zeta)=0\big] \,p_{F(\zeta)}(0) \,\dm(\zeta),
	\label{eq:kacrice}
	\end{equation}
	where $p_{F(\zeta)}(0)$ is the probability density of $F(\zeta)$ at $0$; see, e.g.,  \cite[Th.~6.2]{level}.
	
	We first compute the value
	\begin{align}\label{eq_axxx}
	p_{F(\zeta)}(0) = \frac{1}{\pi} e^{-\abs{\zeta}^2} e^{-\abs{F^1(\zeta)}^2e^{-\abs{\zeta}^2}}.
	\end{align}
	Second, since $F$ is analytic, the determinant in \eqref{eq:kacrice} can easily be seen to simplify to $\abs{\det DF(\zeta)} = \abs{\partial_\zeta F(\zeta)}^2$.
	The joint vector $(F(z), \partial_z F(z))$ has mean $(F^1(z), \partial_z F^1(z))$ and covariance
	\begin{equation}
	\Cov[(F(z), \partial_z F(z))]
	=
	\begin{pmatrix}
	e^{\abs{z}^2} & z e^{\abs{z}^2}\\
	\overline{z} e^{\abs{z}^2} & (1+ \abs{z}^2) e^{\abs{z}^2}
	\end{pmatrix}.
	\end{equation}
	Following a Gaussian regression approach, see, e.g., \cite[Prop.~1.2]{level}, the conditional expectation of $\abs{\partial_\zeta F(\zeta)}^2$ given $F(\zeta)=0$ is the same as the expectation of $\abs{W}^2$, where $W=\partial_\zeta  F^1(\zeta)-\overline{\zeta}F^1(\zeta) + W_0$ and
	$W_0$ is a circularly symmetric complex Gaussian random variable with variance $e^{\abs{\zeta}^2}$ (and zero mean).
	Thus,
	\begin{equation}\label{ayyy}
	\E\big[\abs{\det DF(\zeta)} \, \big\vert \, F(\zeta)=0\big]
	= e^{\abs{\zeta}^2} + \abs{\partial_\zeta F^1(\zeta)-\overline{\zeta}F^1(\zeta)}^2.
	\end{equation}
	Inserting \eqref{ayyy} and \eqref{eq_axxx} into \eqref{eq:kacrice} yields \eqref{eq_rho1}.
\end{proof}

\section{Proof of Theorem \ref{mth}}\label{sec_proof}
We present the proof of Theorem \ref{mth} in several steps. The strategy is two-fold: (i) to show that 
computed zeros are close to true ones, we relate the comparison test
\eqref{eq_test} to a similar property involving non-grid points and apply the minimum principle from Lemma \ref{lemma_min}; (ii) to show that true zeros do trigger a detection, we show that the test \eqref{eq_test} 
is satisfied by linearly approximating the input function. The two objectives are in tension: while a large comparison margin $\eta_\gp$ would facilitate (i) by absorbing possible oscillations between a grid and a close-by non-grid point, a small margin makes the comparison test easier to satisfy and thus facilitates (ii). The core of the proof consists in showing that the adaptive margin \eqref{eq_th} strikes the desired balance with high probability.

Initially we bound the Hausdorff distance between the exact and computed zero sets (showing that each of the sets lies in a small neighborhood of the other). We then refine this conclusion to a bound on the Wasserstein distance by analyzing the sieving step.

\subsection{Preparations}
Let $L$, $\sigma$, $\delta$, and $F$ satisfy the assumptions of the theorem, and denote by $\outseta$ the set produced by the AMN algorithm after the selection step. 
Recall that $L \geq 1$.
By choosing a sufficiently large constant in \eqref{eq_suc_pro}, we can assume that $\delta \leq 1/5$; otherwise, the success probability would be trivial. For the same reason, we can assume that
\begin{align}\label{eq_wlog}
\delta^4 \exp\left(\frac{A^2}{8\noise^2}\right) \leq 1.
\end{align}

\subsection{Excluding bad events}\label{sec_bad}
We 
let $\beconst = 8 \noise$ and wish to
apply Lemma \ref{lemma_lin} with 
$t=\beconst \sqrt{\logi}$. By \eqref{eq_wlog},
\begin{align*}
\delta^4 \exp\left(\frac{A^2}{16\noise^2}\right) \leq
\delta^4 \exp\left(\frac{A^2}{8\noise^2}\right) \leq 1.
\end{align*}
Hence $t \geq A$ and we can apply Lemma \ref{lemma_lin} to conclude that
\begin{align}
\label{eq_aaaa0}
& \abs{F_w(z) - \big(F_w(0) + F_w'(0) z \big)} 
\leq \beconst \cdot \sqrt{\logi} \cdot |z|^2
\leq 2 \beconst \cdot \sqrt{\logi} \cdot |z|_{\infty}^2
,
\\
\label{eq_aaaa1}
& \abs{F_w(z)e^{-\frac12 |z|^2} - \big(F_w(0) + F_w'(0) z \big)}
\leq \beconst \cdot \sqrt{\logi} \cdot |z|^2
\leq 2 \beconst \cdot \sqrt{\logi} \cdot |z|_{\infty}^2
,
\end{align}
for all $w \in \domainplus$ and $|z| \leq 10$, except for an event of probability at most
$C L^2 \exp\big[ -(t-\conf)^2/(8\sigma^2)\big]$,
where $C$ is an absolute constant.
Since $(t-\conf)^2 \geq \frac{t^2}{2}-\conf^2$, we further have 
\begin{align*}
C L^2 \exp\bigg({-}\frac{(t-\conf)^2}{8\sigma^2}\bigg)
&\leq
C L^2 \exp\bigg(\frac{\conf^2}{8\sigma^2}\bigg) \delta^{\frac{\beconst^2}{16\sigma^2}}
= 
C L^2 \exp\bigg(\frac{\conf^2}{8\sigma^2}\bigg) \delta^{4}.
\end{align*}
Second, we select a large absolute constant $\cons>1$ to be specified later, and use Lemma~\ref{lemma_mult}
with
\begin{align*}
\alpha &= \cons \beconst \cdot \sqrt{\logi} \cdot \delta^2,
\\
\beta &= 2 \cons \beconst \cdot \sqrt{\logi} \cdot \delta,
\end{align*}
to conclude that, for each grid point $\gp \in \gridlplus$,
\begin{align}\label{eq_b}
\mbox{either } |F_\gp(0)| > \alpha, \mbox{ or }
|F'_\gp(0)| > \beta,\qquad \mbox{(possibly both)},
\end{align}
except for an event with probability at most
$\lesssim  L^2 \log^2(1/\delta) \delta^4$.

Overall we have excluded events with total probability
\begin{align*}
\lesssim  L^2  \exp\bigg(\frac{\conf^2}{8\sigma^2}\bigg)
\log^2(1/\delta)
\delta^{4}.
\end{align*}
In what follows, we show that under the complementary events the conclusions of Theorem \ref{mth} hold.

\subsection{The true zeros are adequately separated}

We claim that, by taking $\cons$ sufficiently large, the set $\{F=0\} \cap \domainplusdelta$ satisfies:
\begin{align}\label{eq_sep}
\inf \Big\{ |\zeta-\zeta'|_\infty : \zeta, \zeta' \in \{F=0\} \cap \domainplusdelta, \zeta \not= \zeta' \Big\} > 7 \delta.
\end{align}

Suppose that $\zeta, \zeta' \in \{F=0\} \cap \domainplusdelta$ are such that $0<|\zeta-\zeta'|_\infty \leq 7 \delta$.
Since $L/\delta \in \mathbb{N}$, we can select a lattice point $\lambda \in \gridlplus$ such that $0 < |\lambda-\zeta| \leq \delta$.
We now use repeatedly \eqref{eq_aaaa0} and \eqref{eq_aaaa1}. 

First, we use \eqref{eq_aaaa0} with $w=\zeta$
and $z=\zeta'-\zeta$, and note
that $F_\zeta(\zeta'-\zeta)=0$ and $F_{\zeta}(0)=0$
while $|\zeta-\zeta'| \leq \sqrt{2} |\zeta-\zeta'|_{\infty} \leq 7 \sqrt{2}  \delta \leq 10$ to obtain:
\begin{align*}
\abs{F_\zeta'(0)} \abs{\zeta'-\zeta}
\leq \beconst \cdot \sqrt{\logi} \cdot |\zeta'-\zeta|^2.
\end{align*}
Since $\zeta \not= \zeta'$, we conclude:
\begin{align}\label{eq_m1}
\abs{F_\zeta'(0)}
\leq \beconst \cdot \sqrt{\logi} \cdot |\zeta'-\zeta|.
\end{align}
Second, we similarly apply \eqref{eq_aaaa1} with $w=\zeta$ and $z=\lambda-\zeta$, to obtain
\begin{align*}
\abs{F_\zeta(\lambda-\zeta)\cdot e^{-\frac12 |\lambda-\zeta|^2} - F_\zeta'(0) \cdot (\lambda-\zeta)}
\leq \beconst \cdot \sqrt{\logi} \cdot |\lambda-\zeta|^2.
\end{align*}
Combining the last equation with \eqref{eq_m1} yields
\begin{align}\label{eq_m2}
\abs{F_\zeta(\lambda-\zeta)\cdot e^{-\frac12 |\lambda-\zeta|^2}}
\leq \beconst \cdot \sqrt{\logi} 
\cdot \left(|\lambda-\zeta|^2 +
|\zeta'-\zeta| \cdot |\lambda-\zeta| \right).
\end{align}
Third, we apply \eqref{eq_aaaa0} with $w=\lambda$ and $z=\zeta-\lambda$ to obtain
\begin{align}\label{eq_m3}
\abs{F_\lambda(0) + F_\lambda'(0) \cdot (\zeta-\lambda)} 
\leq \beconst \cdot \sqrt{\logi} \cdot |\zeta-\lambda|^2.
\end{align}
Note that $|F_\lambda(0)| = \abs{F_\zeta(\lambda-\zeta)}\cdot e^{-\frac12 |\lambda-\zeta|^2}$. Hence, combining \eqref{eq_m2} and \eqref{eq_m3} we obtain:
\begin{align*}
|F_\lambda(0)| &\leq 
\beconst \cdot \sqrt{\logi} \cdot
\left(|\lambda-\zeta|^2 +
|\zeta'-\zeta| \cdot |\lambda-\zeta| \right),
\\
\abs{F_\lambda'(0)} \cdot|\lambda-\zeta| &\leq 
\beconst \cdot \sqrt{\logi} \cdot
\left(2 |\lambda-\zeta|^2 +
|\zeta'-\zeta| \cdot |\lambda-\zeta| \right).
\end{align*}
Since $0<|\lambda-\zeta| \leq \delta$ and $|\zeta-\zeta'|_{\infty}\leq 7\delta$, we conclude that
\begin{align*}
|F_\lambda(0)| &\leq 
\beconst \cdot \sqrt{\logi} \cdot
\left(|\lambda-\zeta|^2 +
|\zeta'-\zeta| \cdot |\lambda-\zeta| \right)
\leq 11 \cdot \beconst \cdot \delta^2 \cdot \sqrt{\logi},
\\
\abs{F_\lambda'(0)}&\leq 
\beconst \cdot \sqrt{\logi} \cdot
\left(2 |\lambda-\zeta| +
|\zeta'-\zeta| \right) \leq 12 \cdot \beconst \cdot \delta \cdot \sqrt{\logi}.
\end{align*}
Assuming as we may that $\cons > 11$, this contradicts \eqref{eq_b}. Thus, \eqref{eq_sep} must indeed hold.

\subsection{Linearization holds with estimated slopes}
For each $\gp \in \gridl$, we use the notation
\begin{align*}
\tau_{\gp} =  \frac{F_{\gp}(\delta) - F_{\gp}(0)}{\delta},
\end{align*}
and observe that, by \eqref{eq_aaaa0},
\begin{align}\label{eq_aaaa}
\big|
\tau_{\gp} - F_{\gp}'(0) \big| \leq
\beconst \cdot \sqrt{\logi} \cdot \delta.
\end{align}
Combining this with \eqref{eq_aaaa1}, we conclude that for $|z|_{\infty} \leq 2\delta$ and $\gp \in \gridl$,
\begin{align}
\bigabs{F_\gp(z)e^{-\frac12 |z|^2} - \big(F_\gp(0) + \tau_\gp z \big)}
& \leq \bigabs{F_\gp(z)e^{-\frac12 |z|^2} - \big(F_\gp(0) + F_{\gp}'(0) z\big)} + \abs{z} \bigabs{F_{\gp}'(0) - \tau_\gp}
\notag \\
&\leq 
2 \beconst \cdot \sqrt{\logi} \cdot \abs{z}_{\infty}^2 + \abs{z} \cdot \beconst \cdot \sqrt{\logi} \cdot \delta
\label{eq_lina} \\
&\leq 
(8 + 2 \sqrt{2}) \cdot \beconst \cdot \delta^2 \cdot \sqrt{\logi}.
\notag \\
&\leq 
11 \cdot \beconst \cdot \delta^2 \cdot \sqrt{\logi}.\label{eq_lin}
\end{align}

\subsection{After the selection step, each true zero is close to a computed zero}
We show that
\begin{align}
\label{eq_H_1}
\left(\{F=0\} \cap \domain\right) &\subseteq \outseta + Q_{\delta/2}(0).
\end{align}
(Recall that $\outseta$ is the set produced after the selection step, while the cube $Q_\delta(0)$ is defined in Section \ref{sec_not}.)

Let $\zeta \in \domain$ be a zero of $F$.
Since $L/\delta \in \mathbb{N}$, we can find $\gp \in \gridl$ such that
$|\zeta-\gp|_\infty \leq \delta/2$. We show that $\gp \in \outseta$.

Let us first prove that
\begin{align}\label{eq_lb_tau}
|\tau_\gp| 
\geq 
\cons \beconst \sqrt{\logi} \cdot \delta.
\end{align}
Suppose to the contrary that 
$|\tau_\gp| < \cons \beconst \sqrt{\logi} \cdot \delta$. 
We will show that this contradicts \eqref{eq_b}. 
Assuming as we may that $\cons \geq 1$, by
\eqref{eq_aaaa},
\begin{align*}
|F'_\gp(0)| \leq \big(\cons\beconst+ \beconst \big) \sqrt{\logi} \cdot \delta \leq \beta,
\end{align*}
while, by \eqref{eq_aaaa0},
\begin{align*}
|F_\gp(0)| 
& = |F_\gp(0) - F_\gp(\zeta-\gp)|
\\
& \leq 
\bigabs{F_\gp(\zeta-\gp) - \big( F_\gp(0) + F_\gp'(0) (\zeta-\gp)\big)} + \bigabs{F_\gp'(0) (\zeta-\gp)}
\\
& \leq 
\beconst \sqrt{\logi} |\zeta-\gp|^2 + |F_\gp'(0)| |\zeta-\gp|
\\
& \leq 
\beconst  \sqrt{\logi} \bigg(\frac{\sqrt{2}\delta}{2}\bigg)^2 + |F_\gp'(0)| \frac{\sqrt{2}\delta}{2}
\\
&\leq \frac{\beconst}{2} \sqrt{\logi}\delta^2 + \frac{\sqrt{2}}{2} (\cons\beconst+\beconst) \sqrt{\logi} \delta^2 
\\
& = \bigg(\frac{1+\sqrt{2}}{2} +\frac{\sqrt{2}}{2} \cons \bigg) \beconst\sqrt{\logi}\delta^2
\\
& \leq \alpha,
\end{align*}
provided 
$\cons \geq \frac{1+\sqrt{2}}{2-\sqrt{2}}$.
This indeed contradicts \eqref{eq_b}. We conclude that \eqref{eq_lb_tau} holds.

Second, we show that $\gp \in \outseta$ by showing that the test
\eqref{eq_test} is satisfied.
By \eqref{eq_lina}, and since $|\zeta-\gp|_\infty \leq \delta/2$, we have
\begin{align}
	|F_\gp(0)| 
  & \leq 
	\bigabs{F_\gp(\zeta-\gp)e^{-\frac12 |\zeta-\gp|^2} - (F_\gp(0)+ \tau_\gp (\zeta-\gp))} 
  + |\tau_\gp| |\zeta-\gp|
  \notag \\
  & \leq 2 \beconst \sqrt{\logi}  \delta^2 + \frac{\sqrt{2}}{2}|\tau_\gp| \delta.
  \label{eq:boundfgp0}
\end{align}
Choosing $\cons \geq \frac{8}{3-2\sqrt{2}}$, \eqref{eq:boundfgp0} and \eqref{eq_lb_tau} further imply
\begin{align}
	|F_\gp(0)| 
  \leq  \bigg( \frac{3-2\sqrt{2}}{4} +\frac{\sqrt{2}}{2}\bigg)  \lvert \tau_\gp\rvert\, \delta
  = \frac{3}{4}\lvert \tau_\gp\rvert\, \delta.
 \label{eq:f0smallerdeltatau}
\end{align}
Hence,
\begin{align}\label{eq_maxis}
\eta_\gp = \frac{3}{4} |\tau_\gp| \delta.
\end{align}
Let  $\gpp \in \grid$ be an arbitrary lattice point with $|\gpp-\gp|_{\infty}=2\delta$.
By \eqref{eq_lin}, 
\begin{align*}
|F_\gp(\gpp-\gp)|e^{-\frac12 |\gpp-\gp|^2} 
& =
|F_\gp(\gpp-\gp)e^{-\frac12 |\gpp-\gp|^2} - F_\gp(\zeta-\gp)e^{-\frac12 |\zeta-\gp|^2}|
\\
& = 
\big\lvert F_\gp(\gpp-\gp)e^{-\frac12 |\gpp-\gp|^2} - (F_\gp(0) + \tau_\gp (\gpp-\gp))
 \\*
& \quad 
- F_\gp(\zeta-\gp)e^{-\frac12 |\zeta-\gp|^2} + (F_\gp(0) + \tau_\gp (\zeta-\gp)) 
+ \tau_\gp (\gpp-\zeta)\big\rvert
\\
&\geq |\tau_\gp| |\gpp-\zeta| - (11+2) \beconst \sqrt{\logi}  \delta^2
\\
&\geq |\tau_\gp| |\gpp-\zeta|_{\infty} - 13 \beconst \sqrt{\logi}  \delta^2
\\
&\geq |\tau_\gp| (|\gpp-\gp|_{\infty}-|\zeta-\gp|_{\infty}) -  13 \beconst \sqrt{\logi}  \delta^2
\\
&\geq \frac{3}{2}|\tau_\gp| \delta -  13 \beconst \sqrt{\logi}  \delta^2.
\end{align*}
Together with \eqref{eq:boundfgp0}, this implies
\begin{align*}
|F_\gp(\gpp-\gp)|e^{-\frac12 |\gpp-\gp|^2} 
\geq |F_\gp(0)| + \frac{3-\sqrt{2}}{2} |\tau_\gp| \delta - 15 \beconst
\sqrt{\logi}  \delta^2.
\end{align*}
Finally, we use \eqref{eq_lb_tau} to analyze the obtained comparison margin against \eqref{eq_maxis}:
\begin{align*}
&\frac{3-\sqrt{2}}{2} |\tau_\gp| \delta - 15 \beconst
\sqrt{\logi}  \delta^2
\\
&\qquad\geq \frac{3}{4}|\tau_\gp| \delta + \frac{3-2\sqrt{2}}{4} \Big(\cons \beconst \sqrt{\logi}  \delta^2 \Big)
- 15 \beconst
\sqrt{\logi}  \delta^2
\\
&\qquad= \frac{3}{4}|\tau_\gp| \delta + \bigg(\frac{3-2\sqrt{2}}{4} \cons -  15 
\bigg) \beconst\sqrt{\logi} \delta^2
\\
&\qquad \geq \eta_\gp,
\end{align*}
where we fixed the value of $\cons$ so that $\frac{3-2\sqrt{2}}{4} \cons - 15  \geq 0$.
Therefore, the point $\lambda$ passes the selection test \eqref{eq_test} (as formulated in \eqref{eq_testb}), i.e., $\lambda \in \outseta$, as claimed.

\subsection{After the selection step, each computed zero is close to a true zero}
We show that
\begin{align}
\label{eq_H_2}
\outseta &\subseteq \{F=0\} + Q_{2\delta}(0).
\end{align}
Let $\gp \in \outseta$ be a computed zero, and let us find a zero $z$ of $F$ with $|\gp-z|_\infty \leq 2 \delta$.
In terms of the Fock shift $F_\gp$, the success of the test \eqref{eq_test} reads,
\begin{align}\label{eq_test2}
|F_\gp(\gpp)| e^{-\frac12 |\gpp|^2} \geq |F_\gp(0)| + \eta_\gp,
\quad
\mbox{for all } \gpp \in \grid \mbox{ such that } |\gpp|_\infty = 2 \delta;
\end{align}
see \eqref{eq_testb}. For an arbitrary $z \in \mathbb{C}$ with $|z|_\infty = 2 \delta$, we can find a lattice point $\gpp \in \grid$ with $|\gpp|_\infty= 2\delta$
such that $|z-\gpp|=|z-\gpp|_\infty \leq \delta/2$. 
Hence, by~\eqref{eq_lin},
\begin{align*}
|F_\gp(z)e^{-\frac12 |z|^2} - F_\gp(\gpp)e^{-\frac12 |\gpp|^2}|
& \leq 
\bigabs{F_\gp(z)e^{-\frac12 |z|^2} - \big(F_\gp(0) + \tau_\gp z \big)}
\\*
& \quad 
 + \bigabs{F_\gp(\gpp)e^{-\frac12 |\gpp|^2} - \big(F_\gp(0) + \tau_\gp \gpp \big)}
\\*
& \quad 
+ \abs{\tau_\gp} \abs{z- \gpp}
\\
& \leq \tfrac12 |\tau_\gp| \delta + 22 \beconst \cdot  \sqrt{\logi} \delta^2.
\end{align*}
By \eqref{eq_b} and \eqref{eq_aaaa}, 
either $|\tau_\gp|\delta\geq \cons \beconst \sqrt{\logi} \cdot \delta^2$ 
or $|F_\gp(0)| \geq \cons \beconst \sqrt{\logi} \cdot \delta^2 \geq |\tau_\gp|\delta$.
Choosing $\cons \geq 88$ ensures in the first case that 
\begin{align*}
|F_\gp(z)e^{-\frac12 |z|^2} - F_\gp(\gpp)e^{-\frac12 |\gpp|^2}|
\leq 
\frac{3}{4} |\tau_\gp|\delta
& \leq  \eta_\gp
\end{align*}
and in the second case
\begin{align*}
|F_\gp(z)e^{-\frac12 |z|^2} - F_\gp(\gpp)e^{-\frac12 |\gpp|^2}|
\leq 
\frac{3}{4} |F_\gp(0)|
& \leq  \eta_\gp.
\end{align*}
Combining this with \eqref{eq_test2}, we conclude that
\begin{align}\label{eq_test3}
|F_\gp(z)| e^{-\frac12 |z|^2} \geq |F_\gp(0)|,
\quad
\mbox{for all } z \in \mathbb{C} \mbox{ such that } |z|_\infty = 2 \delta.
\end{align}
By Lemma \ref{lemma_min}, there exists $w_\gp \in \mathbb{C}$ with $|w_\gp| \leq 2 \delta$ such that $F_\gp(w_\gp)=0$. This means that $z_\gp:=w_\gp+\lambda$ is a zero of $F$ that satisfies $|z_\gp -\gp|_\infty \leq 2 \delta$, as desired.

\subsection{Definition of the map $\map$}
We now look into the \emph{sieving step} of the AMN algorithm and analyze the final output set $\outset$.

Given $\zeta \in \{F=0\} \cap \domain$ we claim that there exists $\gp \in \outset$ such that $|\zeta - \gp|_\infty \leq 2 \delta$. Suppose to the contrary that
\begin{align}\label{eq_l1}
|\zeta - \gp|_\infty > 2 \delta, \qquad \gp \in \outset.
\end{align}
By \eqref{eq_H_1}, there exists $\gpp \in \outseta$ such that
$|\zeta-\gpp|_\infty \leq \delta/2$. By \eqref{eq_l1}, $\outset \subsetneq \outset \cup \{\gpp\}$. We claim that $\outset \cup \{\gpp\}$ is $5\delta$-separated. For this, it suffices to check that
\begin{align*}
|\gpp-\gp|_\infty > 4\delta, \qquad \gp \in \outset.
\end{align*}
If $\gp \in \outset$, by \eqref{eq_H_2}, there exist $\zeta' \in \{F=0\}$ such that $|\zeta'-\gp|_\infty \leq 2 \delta$. If $\zeta'=\zeta$, then $|\zeta-\gp|_\infty \leq 2 \delta$, contradicting \eqref{eq_l1}. Thus $\zeta \not= \zeta'$, while, $\zeta' \in \outset+Q_{2\delta} \subset \domainplusdelta$. Hence, we use \eqref{eq_sep} to conclude that
\begin{align*}
|\gpp - \gp|_\infty \geq |\zeta-\zeta'|_\infty - |\gpp-\zeta|_\infty - |\gp-\zeta'|_\infty \geq 7 \delta - \delta/2 - 2 \delta > 4\delta.
\end{align*}
Thus, the set is $5\delta$-separated:
\begin{align*}
\inf \Big\{ |\gp-\gp'|_\infty : \gp, \gp' \in \outset \cup \{\gpp\}, \gp \not= \gp' \Big\} \geq 5 \delta,
\end{align*}
contradicting the maximality of $\outset$. It follows that a point $\gp \in \outset$ such that $|\zeta - \gp|_\infty \leq 2 \delta$ must exist. We choose any such point, and define  $\map(\zeta) = \gp$.

\subsection{Verification of the properties of $\map$}
By construction, the map $\map$ satisfies \eqref{eq_dist}. We now show the remaining properties. To show that $\map$ is injective, assume that $\map(\zeta)=\map(\zeta')$. Then, by \eqref{eq_dist},
\begin{align*}
|\zeta-\zeta'|_\infty \leq |\map(\zeta) - \zeta|_ \infty + |\map(\zeta') - \zeta'|_ \infty \leq 4 \delta.
\end{align*}
Hence, by \eqref{eq_sep}, we must have $\zeta=\zeta'$.

Finally, assume that $\gp \in \outset \cap \domainminus$ and use \eqref{eq_H_2} to select a zero $\zeta \in \{F=0\}$ such that $|\zeta-\gp|_\infty \leq 2 \delta$. Then 
$\zeta \in \domain$, and, by \eqref{eq_dist},
\begin{align}
|\map(\zeta) - \gp|_\infty \leq |\map(\zeta) - \zeta|_\infty + |\zeta-\gp|_\infty \leq 4 \delta.
\end{align}
As $\gp, \map(\zeta) \in \outset$ and $\outset$ is $5\delta$-separated (see \eqref{eq_7s}) we conclude that $\gp=\map(\zeta)$, as claimed.

This concludes the proof of Theorem \ref{mth}. \qed

\section{Numerical Experiments}
\label{sec_nums}
In this section we perform a series of tests of the AMN algorithm and compare its performance with MGN and thresholding supplemented with a sieving step~(ST).

\subsection{Simulation}\label{sec_sim}
We first discuss how to simulate samples from the input model \eqref{eq_rf}. To make simulations tractable, we introduce a fast method to draw samples of the Gaussian entire function $F^0$ given by \eqref{eq_gef} on the finite grid \eqref{eq_gridl}. The method is based on the relation between the Bargmann transform and the short-time Fourier transform \eqref{eq_stft} and amounts to discretizing the underlying signal~$f$.

We fix $L>0$, $T>0$, and $\delta > 0$.
For convenience, we further let $\sigma = 1$ and assume that $T\delta^{-1}$ is an integer.
Recall that we also assumed that $L\delta^{-1}$ is an integer.

To model a discretization of $\mathcal{N}$,
we take i.i.d.\ noise samples in the interval $[-T-L, T+L] \subseteq \mathbb{R}$ spaced by a distance $\delta$.
More specifically, we consider a random vector $w=(w_{-(T+L)\delta^{-1}},  \ldots, w_{(T+L)\delta^{-1}})$, where the elements
$w_{s} \sim \mathcal{N}_\mathbb{C}(0, \delta)$ are independent, i.e., $\E[w_s \overline{w_{s}}]=\delta$, and $\E[w_s \overline{w_{s'}}]=0$ for $s \not= s'$.
Here, $w_{s}$ can be interpreted as an integration of $\mathcal{N}$ over the interval $[\delta s, \delta (s+1)]$.

Let $f^1\colon \mathbb{R} \to \mathbb{C}$,
$\genw = g|_{[-T,T]}$ the restriction of $g(t) = (\tfrac{2}{\pi})^{\frac14} \,e^{-t^2}$ to the compact support $[-T, T]$
and define
\begin{align}\label{eq_discrete_stft_white_noise}
\discSTFT (k+ij)
	:= \sum_{s=-T\delta^{-1} + k}^{T\delta^{-1}+k}
					\left(w_s + \delta f^1(\delta s) \right) \overline{\genw(\delta(s-k))}
					e^{-2 i s j \delta^{2} },
\end{align}
for $k,j \in \{-L\delta^{-1}, \dots, L\delta^{-1}\}$.
The mean of $\discSTFT$ is given by
\begin{align*}
		\mathbb{E}[\discSTFT(k+ij)]
		= \delta \hspace{-2.2mm}\sum_{s=-T\delta^{-1} + k}^{T\delta^{-1}+k}\hspace{-2.2mm}
		f^1(\delta s) \overline{\genw\left(\delta(s-k)\right)}
		e^{-2 i s j \delta^{2} },
		\qquad k,j \in \{-L\delta^{-1}, \dots, L\delta^{-1}\},
\end{align*}
and approximates the integral
\begin{equation*}
	\int_{-\infty}^{\infty}  f^1(t) g(t-x) e^{-2 i y t} dt 
  = e^{-i x y} e^{-\frac12 (x^2+y^2)} F^1(\overline{z}),
\end{equation*}
with $x= \delta k$ and $y= \delta j$.
Furthermore, the covariance of $\discSTFT$ is 
\begin{align*}
	& \operatorname{Cov}\big( \discSTFT (k+ij),    \discSTFT (k'+ij') \, \big)
  \\
	& \qquad = \delta
	 	\sum_{s=-T\delta^{-1}+k'}^{T\delta^{-1}+k'}	 \genw	\left(\delta (s-k')	\right)
									 		\overline{\genw	\left(\delta (s-k)	\right)}
						 						e^{-2 i s (j-j') \delta^{2}}.
\end{align*}
For small $\delta$ and sufficiently large $T$, this is an approximation of the integral
\begin{align}\label{eq_covariance}
	 \int_{-\infty}^{\infty} g\big(t- \delta k'\big) \overline{g\big(t-\delta k\big)} e^{-2 i t (j-j') \delta}\, dt
  & = e^{-\frac{u^2+v^2+x^2+y^2}{2}} e^{i(uv-xy)}  e^{(x-iy)(u+iv)},
\end{align}
with $x= \delta k$, $y= \delta j$, $u= \delta k'$, and $v= \delta j'$.
Therefore, if we take $T$ large enough so that
 we can ignore the numerical error introduced by the truncation of the normalized Gaussian window $g$,
 we obtain in \eqref{eq_discrete_stft_white_noise} a random  Gaussian vector whose covariance structure 
 approximates the right-hand side of \eqref{eq_covariance} on the grid $\gridl$, provided that $\delta$ is small.
 
 To obtain a vector whose covariance structure approximates \eqref{eq_gef} we proceed as follows. By conjugating $z$ in \eqref{eq_discrete_stft_white_noise} and multiplying by the deterministic factor $e^{-i x y}$,
we obtain an approximate sampling of \eqref{eq_rf} with weight $e^{-\frac12 |z|^{2}}$:
\begin{equation}\label{eq_wbt}
	e^{-\frac12 |z|^{2}} \bt(z)
	\approx e^{-i x y} \,  \discSTFT (\bar{z})
\end{equation}
for $z=\delta k + i \delta j$. We  carry out all computations with the weighted function \eqref{eq_wbt}, as the unweighted version can lead to floating point arithmetic problems. Note that, for a grid point $\gp$, the comparison margin \eqref{eq_th} can be 
expressed in terms of $e^{-\frac{1}{2}|\, \cdot \,|^{2}} F(\,\cdot\,)$ as
\begin{align*}
	\eta_\gp = 
	\max\big\{
	e^{-\frac{1}{2} | \gp|^2} \abs{F(\gp)}, 
	\tfrac{3}{4}
	\big|
	e^{\frac{1}{2} \delta (2i \operatorname{Im}(\lambda) + \delta )  }
	e^{-\frac{1}{2}|\lambda+\delta|^{2}}
	F(\lambda+\delta) - e^{-\frac{1}{2}|\lambda|^{2}} F(\lambda)
	\big|
	\big\}.
\end{align*}

\pagebreak[1]
\subsection{Specifications for the experiments}
\subsubsection{Implementation of the sieving step in AMN}\label{sec_siev_step}
In order to fully specify the AMN algorithm we need to fix an implementation of the sieving step, which provides a subset $Z \subseteq Z_1$ satisfying \eqref{eq_7s}, and such that no proper superset $Z_1 \supseteq \tilde{Z} \supsetneq Z$ satisfies \eqref{eq_7s}. We choose an implementation that uses knowledge of the input $F$ to decide which points are to be discarded.
We assume that $Z_1$ is non-empty, otherwise $Z$ is trivial.

\noindent\rule{\textwidth}{1pt}
\noindent {\bf Algorithm \sm: \,}{Obtain a maximal subset that is $5\delta$ separated.}\vspace{-0.2cm}\\
\noindent\rule{\textwidth}{1pt}

\alstep{Input}: Values of a function $F$ on a grid $\gridl$. A discrete non-empty set $Z_1\subseteq \gridl$.

\medskip

\alstep{Step 1}: Copy the set $Z_1$ to $Z_1^\text{aux}$. 

\alstep{Step 2}: Consider the (pre)ordered set $(Z_1^\text{aux}, \symbOrder)$, where
\begin{align}\label{eq_order}
	\gp \symbOrder \gpp \quad \iff \quad e^{-\frac12 | \gp |^2 } \abs{F(\gp)} \leq e^{-\frac12 | \gpp |^2 } \abs{F(\gpp)}, \qquad \gp,\gpp\in\gridl.
\end{align}

\alstep{Step 3}: Choose a minimal point $\gp \in (Z_1^\text{aux}, \symbOrder)$.

\alstep{Step 4}: Add $\gp$ to $Z$.

\alstep{Step 5}: Remove all $\gpp \in Z_1^\text{aux}$ such that 
\begin{equation}
\label{eq_6delta}
	0 \leq |\gp-\gpp|_{\infty} \leq 4\delta.
\end{equation}

\alstep{Step 6}: If the set $Z_1^\text{aux}$ is not empty, repeat Steps 3--6. If the set $Z_1^\text{aux}$ is empty, then the algorithm ends.

\alstep{Output}: The set $\outset$.\\
\noindent\rule{\textwidth}{1pt}

The resulting set $Z \subseteq Z_1$ always satisfies \eqref{eq_7s}. Moreover, any superset $\tilde{Z} \supsetneq Z$ included in $Z_1$ must contain some of the discarded points $\gpp \in Z_1$, which by construction satisfy \eqref{eq_6delta} for some $\gp \in Z$, and therefore $\tilde{Z}$ is not $5\delta$ separated. Thus, $Z$ is indeed maximal with respect to  \eqref{eq_7s}.

The choice of $\gp \in Z_1^\text{aux}$ in Step 3 of \sm\ is not essential. 
Our particular choice is motivated by finding the zeros of $F$; however, we did not observe any significant performance difference
when using other algorithms than \sm\ as the sieving step of AMN.

\subsubsection{Specification of the compared algorithms}\label{sec_algos}
Given the values of a function $F\colon \mathbb{C} \to \mathbb{C}$ on the grid
$\gridl$, we consider the following three algorithms to compute an approximation of $\{F=0\} \cap \gridlmo$.
\begin{itemize}[itemsep=0.35cm]
	\item	AMN: the AMN algorithm run with domain length $L-1$ and with sieving step \sm{} implemented as described in Section \ref{sec_siev_step},
	\item 	MGN: outputs the set of all grid points $\lambda \in \gridlmo$ such that
	\begin{equation}\label{eq_mgn}
	e^{-\frac{1}{2}|\gp|^{2}}|F(\lambda)| \leq e^{-\frac{1}{2}|\gpp|^{2}}|F(\mu)|, \quad|\gp-\gpp|_{\infty}=\delta.
	\end{equation} 
	\item	ST: outputs the set of grid points $\lambda \in \gridlmo$ obtained as the result of applying the sieving algorithm \sm{} to 
	\begin{align*}
	\left\{ \gp \in \gridlmo : e^{-\frac12|\gp|^2} |F(\gp)| \leq 2 \delta \right\}.
	\end{align*}
\end{itemize}
Note that each of the algorithms relies only on the samples of $F$ on
$\gridlmoplus$. The use of a common input grid $\gridl$ simplifies the notation when considering various grid spacing parameters $\delta$.

\subsubsection{Varying the grid resolution}\label{sec_resolutions}
In the numerical experiments, we start with a small minimal spacing value $\delta=\delta_{\text{Hi}}$, that provides a high resolution approximation in \eqref{eq_discrete_stft_white_noise}, and simulate $F$ as in Section \ref{sec_sim}. We then incrementally double $\delta$ to produce coarser grid resolutions and subsample $F$ accordingly. More precisely, each element of the grid $\gridl$ can be written as 
\begin{align}
\gp_{k,l} = (-L + k\delta) + i(-L + l\delta), \qquad {\parbox{3cm}{$0\leq k \leq M$,\\ $0\leq l \leq N,$}}
\end{align}
for adequate $M$, $N >0$. If $F$ is given on $\gridl$, we subsample it by setting
\begin{align}\label{eq_subs}
\subs(F)(\gp_{k,l}) := F\big(\gp_{2k +i 2l}\big),
\end{align}
for values $(k,l)$ such that the indices $2k +i 2l$ are valid.

\subsection{Faithfulness of simulation of zero sets}

As a first test, we simulate random inputs from the model \eqref{eq_rf}, as specified in Section \ref{sec_sim}, apply the above-described three different algorithms, and test whether this process faithfully simulates the zero sets of the random function \eqref{eq_rf}. To this end, we estimate first or second order statistics on the computed zero sets by averaging over several realizations of \eqref{eq_rf}, and compare them to the corresponding expected values concerning the zero sets of \eqref{eq_rf}.

\subsubsection{No deterministic signal}
We first consider the case $F^1\equiv 0$ and $\noise=1$ in \eqref{eq_rf}.
Let $\discBargmann_{1}^{\delta_{\text{Hi}}}, \ldots, \discBargmann_{R}^{\delta_{\text{Hi}}}$ be $R$ independent realizations of samples of \eqref{eq_rf} on a grid $\gridl$ with resolution $\delta=\delta_{\text{Hi}}$, simulated as in Section \ref{sec_sim}. These are then subsampled with \eqref{eq_subs} yielding
$F_{r}^{\delta_{k}} = \subs^{(k)}(F_{r}^{\delta_{\text{Hi}}})$ and used as input for AMN, MGN, and ST, as specified in Section \ref{sec_algos}. The corresponding output sets are denoted $\zdm{r}$ where we omit the dependence on the method to simplify the notation. These sets should approximately correspond to $\{F_r=0\} \cap \gridlmo$, for $R$ independent realizations of \eqref{eq_rf}. We now put that statement to test.

The expected number of zeros of the random function $F$ on a Borel set $\Theta \subseteq \bC$ is
\begin{align}\label{eq_zeros_square}
\E[\card{\{F=0\} \cap \Theta}] 
= \int_{\Theta} \frac{1}{\pi} \, \dm(\zeta)
= \frac{\card{\Theta}}{\pi},
\end{align}
see, e.g., \cite[Section 2.4]{gafbook}. We define the following empirical estimator for the \emph{first intensity} $\rho_1=1/\pi$:
\begin{align}\label{eq_estrho1}
  \ei{\Theta}{r}{\delta} 
  = \frac{\card{\zdm{r} \cap \Theta}}{\card{\Theta}}.
\end{align}
If the computed set $\zdm{r}$ were replaced by $\{F=0\}$ in \eqref{eq_estrho1}, the estimator would be unbiased. The mean of the estimation error $\ei{\Theta}{r}{\delta} - 1/\pi$ thus measures the quality of the algorithm used to compute $\zdm{r}$, as it should be close to zero when the algorithm is faithful.
In Table~\ref{table_zero_1000_7_6}, we present the empirical means and the empirical standard deviations of the estimation error over $R=1000$ independent realizations $F_{r}^{\delta}$ for $L=7$, $\Theta = \domainminusone$, $T=6$, and various grid sizes $\delta$.

\begin{table}[tb]
\centering
\caption{Empirical means $\pm$ standard deviations of the estimation errors $\ei{\Theta}{r}{\delta}{} - 1/\pi$ for $\Theta=\domainminusone$, $L=7$, and $1000$ independent realizations.
Benchmark values for a faithful computation are $0$ for the mean and $0.01165$ for  the standard deviation.}\vspace{1mm}
\label{table_zero_1000_7_6}
\begin{tabular}{cccc}
\toprule
$\delta$ & AMN& MGN& ST \\ 
\midrule
$ 2^{-4} $ 	& $ -0.00120 \pm 0.01171 $ & $ -0.00048 \pm 0.01150 $ & $ +0.01868 \pm 0.02858 $ \\ 
$ 2^{-5} $ 	& $ -0.00062 \pm 0.01164 $ & $ -0.00057 \pm 0.01162 $ & $ +0.02189 \pm 0.04047 $ \\ 
$ 2^{-6} $ 	& $ -0.00065 \pm 0.01156 $ & $ -0.00064 \pm 0.01155 $ & $ +0.02280 \pm 0.05391 $ \\ 
$ 2^{-7} $ 	& $ -0.00068 \pm 0.01153 $ & $ -0.00068 \pm 0.01153 $ & $ +0.02354 \pm 0.06774 $ \\ 
$ 2^{-8} $ 	& $ -0.00062 \pm 0.01155 $ & $ -0.00062 \pm 0.01155 $ & $ +0.02424 \pm 0.07429 $ \\ 
$ 2^{-9} $ 	& $ -0.00067 \pm 0.01158 $ & $ -0.00067 \pm 0.01158 $ & $ +0.02390 \pm 0.07237 $ \\ 
\bottomrule
\end{tabular}
\vspace{5mm}
\end{table}

To derive a benchmark for the empirical standard deviation of $\ei{\Theta}{r}{\delta} - 1/\pi$, we express the variance of $\card{\{F=0\} \cap \Theta}/ \card{\Theta}$ in terms of the
\emph{second intensity function} $\rho_2(\zeta, \zeta')$
of $\{F=0\}$ as follows:
\begin{align}\label{eq_var}
  & \E\bigg[\bigg(\card{\{F=0\} \cap \Theta} - \frac{\card{\Theta}}{\pi}\bigg)^2\bigg] 
  \\[1mm]\nonumber
	& \quad 
  = \E\big[\card{\{F=0\} \cap \Theta}\cdot(\card{\{F=0\} \cap \Theta}-1)\big] 
  - \frac{\card{\Theta}^2}{\pi^2} + \frac{\card{\Theta}}{\pi}
  \\[1mm]\nonumber
  & \quad 
  = \int_{\Theta} \int_{\Theta} \rho_2(\zeta, \zeta') \, \dm(\zeta)\, \dm(\zeta')
  - \frac{\card{\Theta}^2}{\pi^2} + \frac{\card{\Theta}}{\pi}.
\end{align}
A formula for $\rho_2(\zeta, \zeta')$ is provided in \cite{Hannay} and numerical integration over $\Theta = \Omega_{L-1}$ results in 
$\sqrt{\Var[\card{\{F=0\} \cap \Theta}/ \card{\Theta}]} \approx 0.01165$.
We see in Table~\ref{table_zero_1000_7_6} that the methods AMN and MGN almost perfectly match the expected mean and standard deviation while ST does not.

\subsubsection{Deterministic signal plus noise}
We now consider the input model \eqref{eq_rf} with
$F^1 \not =0$ and $\noise=1$. We choose $F^1$ from Table \ref{table_fF1} and rescale it so that
$\conf = \sup_{\zeta\in \mathbb{C}} e^{-{\frac12 |\zeta|^2}}|F^1(\zeta)| $
holds for the signal intensities $\conf=1$ and $100$. 

We only test first order statistics of the computed zero sets. The benchmark is provided by Proposition \ref{prop_first}: the expected number of zeros of $F$ in $\Theta$ is
\begin{equation}\label{eq_ek_form}
	\E[\card{\{F=0\} \cap \Theta}] 
	= \int_{\Theta} \rho_1(\zeta) \, \dm(\zeta),
\end{equation}
where $\rho_1$ is given by \eqref{eq_rho1} (with $\noise=1$). 
For each of the tested algorithms, we define an estimator for the error resulting from replacing $\{F=0\}$ in \eqref{eq_ek_form} by the computed set $\zdm{r}$ (for $1\le r\le R$):
\begin{equation}\label{eq_est_nzm}
	\einz{\Theta}{r}{\delta}
	= \frac{\card{\zdm{r} \cap \Theta} - \int_{\Theta} \rho_1(\zeta) \, \dm(\zeta)}{\card{\Theta}}.
\end{equation}
As before, we simulate $R=100$ realizations of $F=F^0+F^1$ on a grid with a certain spacing $\delta$. The empirical average of $\einz{\Theta}{r}{\delta}$ over all realizations is denoted $\widehat{\beta}_{R}(\Theta,\delta)$. As $\rho_1$ is not constant when $F^1 \not= 0$, this time we calculate 
$\widehat{\beta}_{R}(\Theta,\delta)$ on $\Theta=\Omega_{L_1}$ for several values of $L_1$.

The results for $\delta=2^{-9}$ are depicted in Figure \ref{fig:gauss_100_7_6_A_1}. We see that the performance of AMN and MGN is indistinguishable, while ST may perform poorly even at such high resolution. Lower grid resolutions yield similar results.
\begin{table}
	\caption{Functions $f^1$ and their Bargmann transforms $F^1=\mathcal{B}(f)$.}
	\label{table_fF1}
	\begin{tabular}{cc}
		\toprule
		$f^1$ & $F^1$\\ 
    \midrule
		$f^1(t) = \big(\tfrac{2}{\pi})^{\frac14}\,e^{- t^2}$ & $F^1(\zeta)=1$ \\ 
		$f^1(t) =\big(\tfrac{2}{\pi})^{\frac14}\,2te^{- t^2}$ & $F^1(\zeta)=\zeta$ \\ 
    \bottomrule
	\end{tabular}
\vspace{5mm}
\end{table}
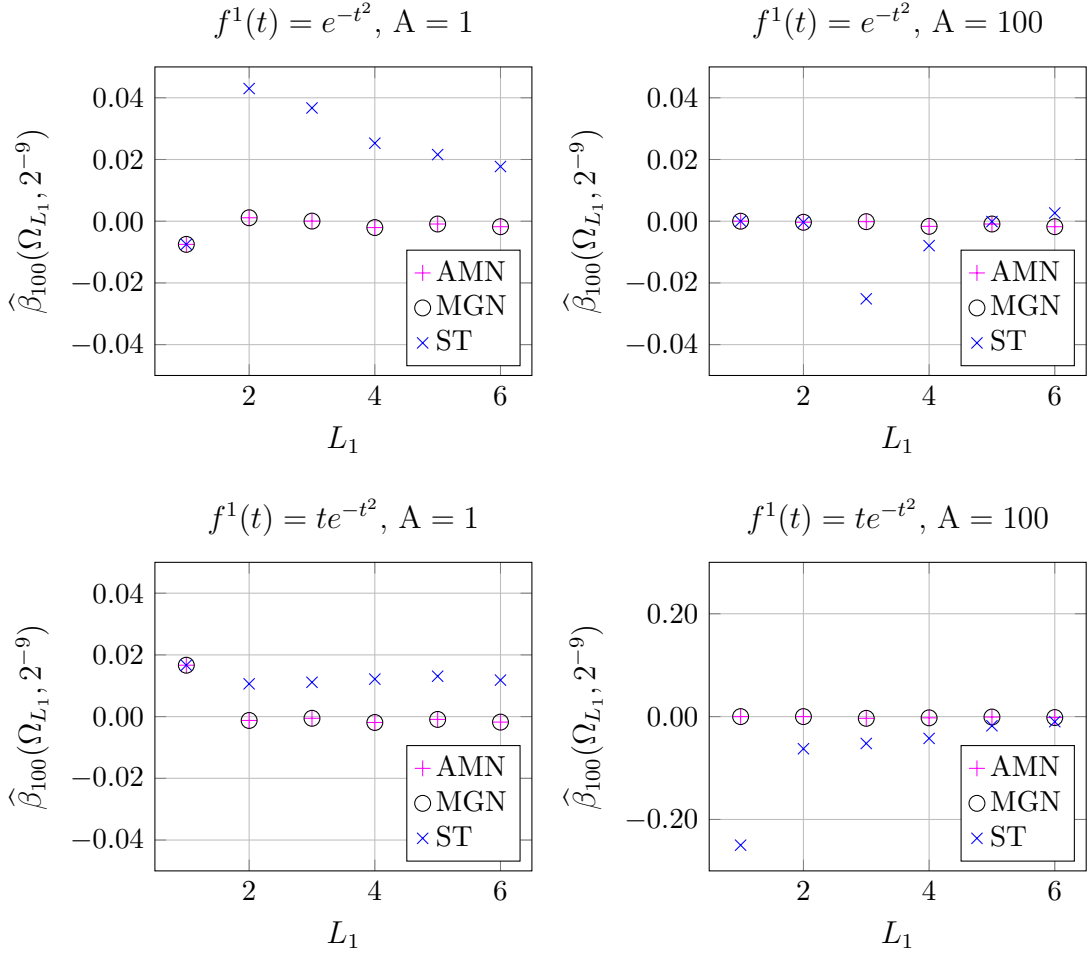
\begin{figure}[tbh]
\centering
\hspace{-5mm}
\begin{minipage}{.45\textwidth}
\begin{tikzpicture}
  \begin{axis}[
    title={$f^1(t) = e^{- t^2}$, $\conf=1$},
    legend entries={AMN,MGN,ST},
    width=\textwidth,
    xmin=0.5,xmax=6.5,
    ymin=-0.05,ymax=0.05,
    grid=minor,
    xlabel={$L_1$},
    ylabel={$\widehat{\beta}_{100}(\Omega_{L_1},2^{-9})$},
    ylabel style={
    },
    tick label style={font=\small},
    grid=major, 
    legend cell align=left,
    legend style={
      legend pos=south east, font=\small
    },
    scaled y ticks = false,
    y tick label style={/pgf/number format/.cd,fixed,fixed zerofill,precision=2},
  ]
\addplot+[octavePink, only marks, mark=+, mark size=3pt] table[x=L, y=AMN] {data/tableEstimator_gauss_100_7_6_A_1.dat}; 
\addplot+[black, only marks, mark=o, mark size=3pt] table[x=L, y=MGN] {data/tableEstimator_gauss_100_7_6_A_1.dat}; 
\addplot+[blue, only marks, mark=x, mark size=3pt] table[x=L, y=ST] {data/tableEstimator_gauss_100_7_6_A_1.dat}; 
\end{axis}
\end{tikzpicture}
\end{minipage}
\hspace{5mm}
\vspace{3mm}
\begin{minipage}{.45\textwidth}
\begin{tikzpicture}
  \begin{axis}[
    title={$f^1(t) = e^{- t^2}$, $\conf=100$},
    legend entries={AMN,MGN,ST},
    width=\textwidth,
    xmin=0.5,xmax=6.5,
    ymin=-0.05,ymax=0.05,
    grid=minor,
    xlabel={$L_1$},
    ylabel={$\widehat{\beta}_{100}(\Omega_{L_1},2^{-9})$},
    ylabel style={
    },
    tick label style={font=\small},
    grid=major, 
    legend cell align=left,
    legend style={
      legend pos=south east, font=\small
    },
    scaled y ticks = false,
    y tick label style={/pgf/number format/.cd,fixed,fixed zerofill,precision=2},
  ]
\addplot+[octavePink, only marks, mark=+, mark size=3pt] table[x=L, y=AMN] {data/tableEstimator_gauss_100_7_6_A_100.dat}; 
\addplot+[black, only marks, mark=o, mark size=3pt] table[x=L, y=MGN] {data/tableEstimator_gauss_100_7_6_A_100.dat}; 
\addplot+[blue, only marks, mark=x, mark size=3pt] table[x=L, y=ST] {data/tableEstimator_gauss_100_7_6_A_100.dat}; 
\end{axis}
\end{tikzpicture}
\end{minipage}
\\
\hspace{-5mm}
\begin{minipage}{.45\textwidth}
\begin{tikzpicture}
  \begin{axis}[
    title={$f^1(t) = t e^{- t^2}$, $\conf=1$},
    legend entries={AMN,MGN,ST},
    width=\textwidth,
    xmin=0.5,xmax=6.5,
    ymin=-0.05,ymax=0.05,
    grid=minor,
    xlabel={$L_1$},
    ylabel={$\widehat{\beta}_{100}(\Omega_{L_1},2^{-9})$},
    ylabel style={
    },
    tick label style={font=\small},
    grid=major, 
    legend cell align=left,
    legend style={
      legend pos=south east, font=\small
    },
    scaled y ticks = false,
    y tick label style={/pgf/number format/.cd,fixed,fixed zerofill,precision=2},
  ]
\addplot+[octavePink, only marks, mark=+, mark size=3pt] table[x=L, y=AMN] {data/tableEstimator_herm1_100_7_6_A_1.dat}; 
\addplot+[black, only marks, mark=o, mark size=3pt] table[x=L, y=MGN] {data/tableEstimator_herm1_100_7_6_A_1.dat}; 
\addplot+[blue, only marks, mark=x, mark size=3pt] table[x=L, y=ST] {data/tableEstimator_herm1_100_7_6_A_1.dat}; 
\end{axis}
\end{tikzpicture}
\end{minipage}
\hspace{5mm}
\begin{minipage}{.45\textwidth}
\begin{tikzpicture}
  \begin{axis}[
    title={$f^1(t) = t e^{- t^2}$, $\conf=100$},
    legend entries={AMN,MGN,ST},
    width=\textwidth,
    xmin=0.5,xmax=6.5,
    ymin=-0.3,ymax=0.3,
    grid=minor,
    xlabel={$L_1$},
    ylabel={$\widehat{\beta}_{100}(\Omega_{L_1},2^{-9})$},
    ylabel style={
    },
    tick label style={font=\small},
    grid=major, 
    legend cell align=left,
    legend style={
      legend pos=south east, font=\small
    },
    scaled y ticks = false,
    y tick label style={/pgf/number format/.cd,fixed,fixed zerofill,precision=2},
  ]
\addplot+[octavePink, only marks, mark=+, mark size=3pt] table[x=L, y=AMN] {data/tableEstimator_herm1_100_7_6_A_100.dat}; 
\addplot+[black, only marks, mark=o, mark size=3pt] table[x=L, y=MGN] {data/tableEstimator_herm1_100_7_6_A_100.dat}; 
\addplot+[blue, only marks, mark=x, mark size=3pt] table[x=L, y=ST] {data/tableEstimator_herm1_100_7_6_A_100.dat}; 
\end{axis}
\end{tikzpicture}
\end{minipage}
\caption{
Empirical mean of $\einz{\Theta}{r}{\delta}$ for different choices of $f^1$ and $\conf$, increasing domain $\Theta = \Omega_{L_1}$ for $L_1<L$, and the three methods. Note the different scale in the bottom right plot illustrating a systematic error in the ST method.}
\label{fig:gauss_100_7_6_A_1}
\end{figure} 

\subsection{Failure probabilities and consistency as resolution decreases}

Having tested the statistical properties of the computed zero sets under the input model \eqref{eq_rf} we now look into the accuracy of the computation for an individual realization $F$. We aim to test the existence of a map as in Theorem \ref{mth}, that assigns true zeros to computed ones with small distortion and almost bijectively. As a proxy for the (unavailable) ground truth $\{F=0\}$ we will use the output of AMN from data at very high resolution (computations with MGN yield indistinguishable results). We thus conduct a \emph{consistency experiment}, where the zero set of the same realization of $F$ is computed from samples on grids of different resolution, and the existence of a map as in Theorem \ref{mth} between both outputs is put to test.

Suppose that samples of a function $F$ are simulated on a high-resolution grid
$\gridl$ with spacing $\delta=\delta_{\text{Hi}}$ and restricted to the low-resolution grid
$\gridl$ with spacing $\delta=\delta_{\text{Lo}}$ by subsampling. 
We compute $\zr{\mbox{}}{\delta_{\text{Hi}}}\subseteq  \domainminusone$ from the high-resolution data using AMN, and $\zd{\mbox{}}{\delta_{\text{Lo}}}\subseteq  \domainminusone$ from the low-resolution data, using one of the algorithms described in Section \ref{sec_algos}. 

Second we construct a set 
$U \subseteq \zr{\mbox{}}{\delta_{\text{Hi}}}$ and a map $\phi\colon U  \rightarrow \zd{r}{\delta_{\text{Lo}}}$ with the following greedy procedure:

\pagebreak[1]
\noindent\rule{\textwidth}{1pt}
\noindent {\bf Construction of $U$ and $\phi$}\vspace{-0.2cm}\\
\noindent\rule{\textwidth}{1pt}

\alstep{Input}: Two subsets of $\domainminusone$: $\zr{\mbox{}}{\delta_{\text{Hi}}}$ and $\zd{\mbox{}}{\delta_{\text{Lo}}}$.

\medskip

\alstep{Step 1}: Choose a total order on $\zr{\mbox{}}{\delta_{\text{Hi}}}$. Let
$U$ and $U'$ be empty sets. If $\zr{\mbox{}}{\delta_{\text{Hi}}}$ is empty, output $U=\emptyset$ and $\phi=\emptyset$. Otherwise proceed to Step 2.

\medskip

\alstep{Step 2}: Let $\gp$ be the first element of $\zr{\mbox{}}{\delta_{\text{Hi}}} \setminus \left(U \cup U' \right)$.

\medskip

\alstep{Step 3}: Let
\begin{equation*}
\Phi(\gp)
= \big\{ \gpp \in \zd{\mbox{}}{\delta_{\text{Lo}}} \setminus \phi(U) : |\gp-\gpp|_{\infty} \leq 2 \delta_{\mathrm{Lo}}\big\}.
\end{equation*}
\noindent If $\Phi(\gp)$ is non-empty, add $\gp$ to $U$, and choose $\phi(\gp)\in \Phi(\gp)$ such that 
\begin{equation*}
\big|\gp - \phi(\gp)\big|_{\infty} 
= \min_{\gpp \in \Phi(\gp)} 	\big|\gp - \gpp\big|_{\infty}.
\end{equation*}
\noindent If $\Phi(\gp)$ is empty, add $\gp$ to $U'$.

\medskip

\alstep{Step 4}: If $\zr{\mbox{}}{\delta_{\text{Hi}}} \setminus \left(U \cup U' \right)$ is non-empty, repeat Steps 2--4.

\medskip

\alstep{Output}: The set $U$ and the map $\phi$.\\
\noindent\rule{\textwidth}{1pt}
The resulting function $\phi$ is injective and satisfies
\begin{align*}
|\phi(\lambda) - \lambda|_{\infty} \leq 2\delta_{\text{Lo}}.
\end{align*}
We say that the computation of $\zd{r}{\delta_{\text{Lo}}}$ was \emph{certified to be accurate} if
\begin{align}\label{eq_suc}
\zr{r}{\delta_{\text{Hi}}} \subseteq U
\quad \mbox{ and } \quad 
\zd{r}{\delta_{\text{Lo}}} 
\cap \Omega_{(L-1)-2\delta_{\text{Lo}}} \subseteq \phi(U).
\end{align}
In this case, the map $\phi$ satisfies properties analogous to the ones in Theorem \ref{mth}. Conceivably, other such maps may exist even if the one constructed in the greedy fashion fails to satisfy \eqref{eq_suc}. We define the following \emph{computation certificate}:
\begin{align*}
\mathcal{M}(\zr{\mbox{}}{\delta_{\text{Hi}}}, \zd{\mbox{}}{\delta_{\text{Lo}}}) = \begin{cases}
	0 & \mbox{if }\eqref{eq_suc} \mbox{ holds} \\
	1 & \mbox{otherwise.}
	\end{cases}
\end{align*}

The experiment to estimate failure probabilities as a function of the grid resolutions is fully specified as follows. 
We consider the input model \eqref{eq_rf} with $\noise=1$. We choose $F^1$ from Table \ref{table_fF1} and rescale it so that $\conf = \sup_{\zeta\in\mathbb{C}} e^{-\frac12 |\zeta|^2}|F^1(\zeta)| $ holds for the signal intensities $\conf=1$ and $100$.
We fix $L>0$ and $\delta_{\text{Hi}}>0$ and let $\discBargmann_{1}^{\delta_{\text{Hi}}}, \ldots, \discBargmann_{R}^{\delta_{\text{Hi}}}$ be $R$ independent realizations of samples of \eqref{eq_rf} on a grid $\gridl$ with resolution $\delta=\delta_{\text{Hi}}$, simulated as in Section~\ref{sec_sim}. 
These are then subsampled $j$ times with \eqref{eq_subs} yielding
$F_{r}^{\delta_{k}} = \subs^{(k)}(F_{r}^{\delta_{\text{Hi}}})$, $1 \leq k \leq j$.

We use AMN with input $F_{r}^{\delta_{\text{Hi}}}$ to obtain a set $\zr{r}{\delta_{\text{Hi}}}$.
Further, for each $1 \leq k \leq j$, we use each of the algorithms $M = \text{AMN, MGN, or ST}$ with input $F_{r}^{\delta_k}$ to obtain sets $\zd{r,M}{\delta^k}$.
Finally, we compute all the certificates $\mathcal{M}(\zr{r}{\delta_{\text{Hi}}}, \zd{r,M}{\delta_k})$ and average them over all realizations to obtain the following \emph{estimated upper bound for the failure probability of the method $M$ with grid spacing $\delta=\delta_k$}:
\begin{equation}\label{eq_estim}
p(\delta_k,M) := \frac{1}{R} \sum_{r=1}^R \mathcal{M}(\zr{r}{\delta_{\text{Hi}}}, \zd{r,M}{\delta_k}).
\end{equation}

We present in Table \ref{table_pfal} values obtained for $p(\delta_k,M)$ 
 for a resolution starting as high as $\delta_{\text{Hi}}=2^{-9}$, with a truncation of the window $g$ at $T=6$, in the target domain $\domainminusone$ for $L=7$, and $R=1000$ realizations of a zero-mean $\bt$. 
We also present the results for $F^1$ as in Table \ref{table_fF1}, rescaled to achieve a signal intensity $\conf=1$ or $\conf=100$.
We see that both AMN and MGN deliver very low failure probabilities (with MGN slightly outperforming AMN at lower resolutions). In contrast, ST delivers large failure probabilities even at high resolution.

\begin{table}[tbh] 
\centering 
\caption{Estimation of the failure probability $p(\delta_k,M)$
	in the sense of Theorem \ref{mth}, in the domain $\domainminusone$ with parameters $\delta_{\text{Hi}}=2^{-9}$, $T=6$, and $L=7$. Averages are computed over $R=1000$ and $R=100$ realizations for the pure noise and signal $f^1$ plus noise cases, respectively.}\vspace{1mm}
\label{table_pfal}  
\resizebox{\textwidth}{!}{%
\begin{tabular}{@{\extracolsep{4pt}}cccccccccccccccc} 
\toprule 
& \multicolumn{3}{c}{$f^1=0$} & \multicolumn{6}{c}{$f^1=\exp (-t^2)$} & \multicolumn{6}{c}{$f^1= t \exp (-t^2)$}\\
\midrule
& \multicolumn{3}{c}{} & \multicolumn{3}{c}{$\conf=1$}  & \multicolumn{3}{c}{$\conf=100$}  & \multicolumn{3}{c}{$\conf=1$}  & \multicolumn{3}{c}{$\conf=100$} \\
\cmidrule{5-7} 
\cmidrule{8-10} 
\cmidrule{11-13} 
\cmidrule{14-16} 
$\delta$ & AMN& MGN& ST & AMN& MGN& ST & AMN& MGN& ST & AMN& MGN& ST & AMN& MGN& ST\\ 
\midrule
$ 2^{-4} $ 	& $ 0.082 $ & $ 0.001 $ & $ 0.665 $  	& $ 0.07 $ & $ 0.00 $ & $ 0.67 $  	& $ 0.13 $ & $ 0.00 $ & $ 0.87 $  	& $ 0.07 $ & $ 0.00 $ & $ 0.64 $  	& $ 0.18 $ & $ 0.00 $ & $ 1.00 $  \\ 
$ 2^{-5} $ 	& $ 0.007 $ & $ 0.000 $ & $ 0.536 $  	& $ 0.00 $ & $ 0.00 $ & $ 0.50 $  	& $ 0.00 $ & $ 0.00 $ & $ 0.75 $  	& $ 0.00 $ & $ 0.00 $ & $ 0.52 $  	& $ 0.01 $ & $ 0.00 $ & $ 1.00 $  \\ 
$ 2^{-6} $ 	& $ 0.001 $ & $ 0.000 $ & $ 0.419 $  	& $ 0.00 $ & $ 0.00 $ & $ 0.41 $  	& $ 0.00 $ & $ 0.00 $ & $ 0.70 $  	& $ 0.00 $ & $ 0.00 $ & $ 0.42 $  	& $ 0.00 $ & $ 0.00 $ & $ 1.00 $  \\ 
$ 2^{-7} $ 	& $ 0.000 $ & $ 0.000 $ & $ 0.389 $  	& $ 0.00 $ & $ 0.00 $ & $ 0.32 $  	& $ 0.00 $ & $ 0.00 $ & $ 0.72 $  	& $ 0.00 $ & $ 0.00 $ & $ 0.34 $  	& $ 0.00 $ & $ 0.00 $ & $ 1.00 $  \\ 
$ 2^{-8} $ 	& $ 0.000 $ & $ 0.000 $ & $ 0.369 $  	& $ 0.00 $ & $ 0.00 $ & $ 0.29 $  	& $ 0.00 $ & $ 0.00 $ & $ 0.65 $  	& $ 0.00 $ & $ 0.00 $ & $ 0.33 $  	& $ 0.00 $ & $ 0.00 $ & $ 1.00 $  \\ 
$ 2^{-9} $ 	& $ 0.000 $ & $ 0.000 $ & $ 0.359 $  	& $ 0.00 $ & $ 0.00 $ & $ 0.31 $  	& $ 0.00 $ & $ 0.00 $ & $ 0.71 $  	& $ 0.00 $ & $ 0.00 $ & $ 0.32 $  	& $ 0.00 $ & $ 0.00 $ & $ 1.00 $  \\ 
\bottomrule 
\end{tabular}%
}
\end{table}

\section{Conclusions and outlook}\label{sec_con}
We analyzed the AMN algorithm under a stochastic input model aimed to describe the performance of the method in practice \cite{spielman2009smoothed}. One limitation of our analysis is the assumption that grid samples of the Bargmann transform are exactly given, while, more realistically, acquired data corresponds to averages of the signal values resulting from analog to digital conversion and numerical integration. Second, we considered complex-valued white noise, while in practice noise may also be colored or real-valued. We understand that the techniques used to prove Theorem \ref{mth} are general enough to allow for a refinement of the result in these directions. Similarly, we expect to be able to adapt our analysis of AMN to other ensembles of analytic functions, which are relevant in connection to other signal transforms. A more challenging open direction is the investigation of rigorous performance guarantees for MGN, which remains the algorithm of choice in practice.


\end{document}